\documentclass[11pt, a4paper]{article}

\usepackage{graphicx}
\usepackage{color}
\usepackage{amsmath}
\usepackage{amssymb}
\usepackage{amscd}
\usepackage{bbm}

\usepackage{tikz}
\newlength\Colsep
\usetikzlibrary{patterns}

\usepackage{hyperref}
\hypersetup{
    bookmarks=true,         
    colorlinks=true,       
    linkcolor=red,          
    citecolor=green,        
    filecolor=magenta,      
    urlcolor=cyan           
}

\usepackage[utf8]{inputenc}
\usepackage{amsthm}
\usepackage{bbm} 
\usepackage[linesnumbered,lined,boxed,commentsnumbered]{algorithm2e}

\usepackage{marginnote}

\newtheorem{Theorem}{Theorem}
\newtheorem{Assumption}{Assumption}
\newtheorem{Definition}{Definition}
\newtheorem{Lemma}{Lemma}
\newtheorem{Proposition}{Proposition}

\newcommand{\inr}[1]{\bigl< #1 \bigr>}

\newcommand{\norm}[1]{\left\|#1\right\|}%

\newcommand\eps{\epsilon}

\newcommand{\beginproof}{{\bf Proof. {\hspace{0.2cm}}}}
\def \endproof
{{\mbox{}\nolinebreak\hfill\rule{2mm}{2mm}\par\medbreak}}

\DeclareMathOperator*{\argmin}{argmin}

\def\ds1{\textrm{1\kern-0.25emI}} 

\newcommand \E{\mathbb{E}}
\newcommand \R{\mathbb{R}}

\newcommand \cB{{\cal B}}

\newcommand \cD{{\cal D}}

\newcommand \cF{{\cal F}}

\newcommand \cI{{\cal I}}

\newcommand \cK{{\cal K}}
\newcommand \cL{{\cal L}}

\newcommand \cN{{\cal N}}
\newcommand \cO{{\cal O}}

\newcommand \cX{{\cal X}}
\newcommand \cY{{\cal Y}}

\newcommand \bE{{\mathbb E}}

\newcommand \bP{{\mathbb P}}

\newcommand \bR{{\mathbb R}}

\newcommand \bX{{\mathbb X}}

\newcommand{\med}[1]{{\rm Median}(#1)}

\newcommand{\MOM}[2]{\text{MOM}_{#1}\big(#2\big)}

\begin{document}
\title{Robust Statistical learning with Lipschitz and convex loss functions}
\author{Geoffrey Chinot\footnote{geoffreychinot@gmail.com}, Guillaume Lecu{\'e}\footnote{lecueguillaume@gmail.com} and Matthieu Lerasle\footnote{lerasle@gmail.com}}
\maketitle

\begin{abstract}
	We obtain estimation and excess risk bounds for Empirical Risk Minimizers (ERM) and minmax Median-Of-Means (MOM) estimators based on loss functions that are both Lipschitz and convex.  
Results for the ERM are derived under weak assumptions on the outputs and subgaussian assumptions on the design as in \cite{pierre2017estimation}.
The difference with \cite{pierre2017estimation} is that the global Bernstein condition of this paper is relaxed here into a local assumption. 
We also obtain estimation and excess risk bounds for minmax MOM estimators under similar assumptions on the output and only moment assumptions on the design. Moreover, the dataset may also contains outliers in both inputs and outputs variables without deteriorating the performance of the minmax MOM estimators.  

Unlike alternatives based on MOM's principle \cite{lecue2017robust,LugosiMendelson2016}, the analysis of minmax MOM estimators is not based on the small ball assumption (SBA) of \cite{MR3431642}. 
In particular, the basic example of non parametric statistics where the learning class is the linear span of localized bases, that does not satisfy SBA  \cite{saumard2018optimality} can now be handled. 
Finally, minmax MOM estimators are analysed in a setting where the local Bernstein condition is also dropped out. 
It is shown to achieve excess risk bounds with exponentially large probability under minimal assumptions insuring only the existence of all objects. 
\end{abstract}


\section{Introduction}\label{sec:Intro}
In this paper, we  study learning problems where the loss function is simultaneously Lipschitz and convex. 
This situation happens in classical examples such as quantile, Huber and $L_1$ regression or logistic and hinge classification \cite{MR3526202}. 
As the Lipschitz property allows to make only weak assumptions on the outputs, these losses have been quite popular in robust statistics \cite{huber2011robust}. Empirical risk minimizers (ERM) based on Lipschitz losses such as the Huber loss have received recently an important attention \cite{fan,elsener2016robust,pierre2017estimation}. 

Based on a dataset $\{(X_i,Y_i):i=1,\ldots, N\}$ of points in $\cX\times \cY$, a class $F$ of functions and a risk function $R(\cdot)$ defined on $F$, the statistician want to estimate an oracle $f^*\in\argmin_{f\in F} R(f)$ or to predict an output $Y$ at least as good as $f^*(X)$. 
The risk function $R(\cdot)$ is often defined as the expectation of a loss function $\ell:(f,x,y)\in F\times \cX\times \cY\to \ell_f(x,y)\in\R$ with respect to the unknown distribution $P$ of a random variable $(X,Y)\in\cX\times\cY$: $R(f)=\E \ell_f(X,Y)$. 
Hereafter, the risk is assumed to have this form for a loss function $\ell$ such that, for any $(f,x,y)$, $\ell_f(x,y)=\bar{\ell}(f(x),y)$, for some function $\bar{\ell}:\bar{\cY}\times\cY\to\R$, where the set $\bar{\cY}$ is a convex set containing all possible values of $f(x)$.
The loss function $\ell$ is said Lipschitz and convex when the following assumption holds.
\begin{Assumption}\label{assum:lip_conv}
	There exists $L>0$ such that, for any $y \in \cY$, $\bar{\ell}(\cdot,y)$ is $L$-Lipschitz and convex.
\end{Assumption}
Many classical loss functions satisfy Assumption~\ref{assum:lip_conv}  and we recall some of them below.
\begin{itemize}
	\item The \textbf{logistic loss} defined, for any $u\in\bar{\cY}=\R$ and $y\in\cY=\{-1, 1\}$, by 
	$\bar{\ell}(u,y) = \log(1+\exp(-yu))$ satisfies Assumption~\ref{assum:lip_conv} with $L=1$.
	\item The \textbf{hinge loss} defined, for any $u\in\bar{\cY}=\R$ and $y\in\cY=\{-1, 1\}$, by $\bar{\ell}(u,y) = \max(1-uy,0)$ satisfies Assumption~\ref{assum:lip_conv} with $L=1$.
	\item  The \textbf{Huber loss} defined, for any $\delta >0$, $u,y\in\cY=\bar{\cY}=\R$, by
	\[
	\bar{\ell}(u,y) =  
	\begin{cases}
	\frac{1}{2}(y-u)^2&\text{ if }|u-y| \leq \delta\\
	\delta|y-u|-\frac{\delta^2}{2}&\text{ if }|u-y| > \delta
	\end{cases}\enspace,
	\]
	satisfies  Assumption~\ref{assum:lip_conv} with $L=\delta$.
	\item The \textbf{quantile loss} is defined, for any $\tau \in (0,1)$, $u,y\in\cY=\bar{\cY}=\bR$, by
	$\bar{\ell}(u,y)  = \rho_{\tau}(u-y)$ where, for any $z\in \R$, $\rho_{\tau}(z) = z(\tau-I \{ z \leq 0\})$. It satisfies Assumption~\ref{assum:lip_conv} with $L=1$. For $\tau = 1/2$, the quantile loss is the $L_1$ loss.
\end{itemize}
All along the paper, the following assumption is also granted.
\begin{Assumption}\label{assum:convex}
	The class $F$ is convex.
\end{Assumption}


When $(X,Y)$ and the data $((X_i,Y_i))_{i=1}^N$ are independent and identically distributed (i.i.d.), for any $f\in F$, the empirical risk $R_N(f)=(1/N)\sum_{i=1}^N\ell_f(X_i, Y_i)$ is a natural estimator of $R(f)$. The empirical risk minimizers (ERM) \cite{Vapnik:1995:NSL:211359} obtained by minimizing $f\in F\to R_N(f)$ are expected to be close to the oracle $f^*$. 
This procedure and its regularized versions have been extensively studied in learning theory \cite{koltchinskii2011empirical}. 
When the loss is both convex and Lipschitz, results have been obtained in practice \cite{bach2011convex,bubeck2015convex} and theory \cite{MR3526202}.  
Risk bounds with exponential deviation inequalities for the ERM can be obtained under weak assumptions on the outputs $Y$, but stronger assumptions on the design $X$. 
Moreover, fast rates of convergence \cite{MR2051002} can only be obtained under margin type assumptions such as the Bernstein condition \cite{MR2240689,MR3526202}.

The Lipschitz assumption and global Bernstein conditions (that hold over the entire $F$ as in \cite{pierre2017estimation}) imply boundedness in $L_2$-norm of the class $F$, see the discussion preceding Assumption~\ref{assum:fast_rates} for details. 
This boundedness is not satisfied in linear regression with unbounded design so the results of \cite{pierre2017estimation} don't apply to this basic example such as linear regression with a Gaussian design.
To bypass this restriction, the global condition is relaxed into a ``local'' one as in \cite{elsener2016robust,MR3526202}, see Assumption~\ref{assum:fast_rates} below.

The main constraint in our results on ERM is the assumption on the design.
This constraint can be relaxed by considering alternative estimators based on the ``median-of-means" (MOM) principle of \cite{MR702836, MR762855, MR855970, MR1688610} and the minmax procedure of \cite{MR2906886,MR3595933}. 
The resulting minmax MOM estimators have been introduced in \cite{lecue2017robust} for least-squares regression as an alternative to other MOM based procedures \cite{LugosiMendelson2016, LugosiMendelson2017, LugosiMendelson2017-2,lecue2017learning}.
In the case of convex and Lipschitz loss functions, these estimators satisfy the following properties 1) as the ERM, they are efficient under weak assumptions on the noise 2) they achieve optimal rates of convergence under weak stochastic assumptions on the design and 3) the rates are not downgraded by the presence of some outliers in the dataset.

These improvements of MOM estimators upon ERM are not surprising.
For univariate mean estimation, rate optimal sub-Gaussian deviation bounds can be shown under minimal $L_2$ moment assumptions for MOM estimators \cite{devroye2016sub} while the empirical mean needs each data to have sub-Gaussian tails to achieve such bounds \cite{MR3052407}. 
In least-squares regression, MOM-based estimators \cite{LugosiMendelson2016, LugosiMendelson2017, LugosiMendelson2017-2,lecue2017learning,lecue2017robust} inherit these properties, whereas the ERM has downgraded statistical properties under moment assumptions (see Proposition~1.5 in \cite{lecue2016performance}). Furthermore, MOM procedures are resistant to outliers: results hold in the ``$\cO\cup\cI$'' framework of \cite{lecue2017learning,lecue2017robust}, where 
	inliers or informative data (indexed by $\cI$) only satisfy weak moments assumptions and the dataset may contain outliers (indexed by $\cO$) on which no assumption is made, see Section~\ref{sec:OUI}. 
	This robustness, that almost comes for free from a technical point of view is another important advantage of MOM estimators compared to ERM in practice. Figure~\ref{fig:MOMVsERMLogReg} \footnote{All figures can be reproduced from the code available at \url{https://github.com/lecueguillaume/MOMpower}}  illustrates this fact, showing that statistical performance of the standard logistic regression are strongly affected by a single corrupted observation, while the minmax MOM estimator maintains good statistical performance even with $5\%$ of corrupted data. 
\begin{figure}[h]
  \center
  \includegraphics[scale=0.7]{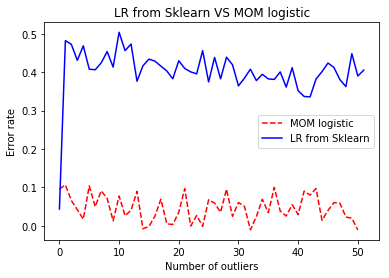}
  \caption{MOM Logistic Regression VS Logistic regression from Sklearn\label{fig:MOMVsERMLogReg} ($p=50$ and $N = 1000$)}
\end{figure}

Compared to \cite{LugosiMendelson2016, lecue2017robust}, considering convex-Lipschitz losses instead of the square loss allows to simplify simultaneously some assumptions and the presentation of the results for MOM estimators: $L_2$-assumptions on the noise in \cite{LugosiMendelson2016, lecue2017robust} can be removed and complexity parameters driving risk of ERM and MOM estimators only involve a single stochastic linear process, see Eq.~\eqref{def:function_r} and~\eqref{comp:par} below. 
Also, contrary to the analysis in least-squares regression, the small ball assumption \cite{MR3431642,mendelson2014learning} is not required here. 
Recall that this assumption states that there are absolute constants $\kappa$ and $\beta$ such that, for all $f\in F$, $\bP[|f(X) - f^*(X)|\geq \kappa \norm{f-f^*}_{L_2}]\geq \beta$. 
It is interesting as it involves only moments of order $1$ and $2$ of the functions in $F$. 
However, it does not hold with absolute constants in classical frameworks such as histograms, see \cite{saumard2018optimality,jon} and Section~\ref{app:ass}. 

Finally,  minmax MOM estimators are studied in a framework where the Bernstein condition is dropped out. In this setting, they are shown to achieve an oracle inequality with exponentially large probability (see Section~\ref{sec:learning_without_assumption}). 
The results are slightly weaker in this relaxed setting: the excess risk is bounded but not the $L_2$ risk and the rates of convergence are ``slow" in $1/\sqrt{N}$ in general. 
Fast rates of convergence in $1/N$ can still be recovered from this general result if a local Bernstein type condition is satisfied though, see Section~\ref{sec:learning_without_assumption} for details.
This last result shows that minmax MOM estimators can be safely used with Lipschitz and convex losses, assuming only that inliers data are independent with enough finite moments to give sense to the results. 

To  approximate minmax MOM estimators, an algorithm inspired from \cite{lecue2017robust, Lecue2018a} is also proposed. 
Asymptotic convergence of this algorithm has been proved in \cite{Lecue2018a} under strong assumptions, but, to the best of our knowledge, convergence rates have not been established.
Nevertheless, the simulation study presented in Section~\ref{sec:simu} shows that it has good robustness performances.

The paper is organized as follows. Optimal results for the ERM are presented in Section~\ref{sec_ERM}.
 Minmax MOM estimators are introduced and analysed in Section~\ref{sec:OUI} under a local Bernstein condition and in Section~\ref{sec:learning_without_assumption} without the Bernstein condition. 
A discussion of the main assumptions is provided in Section~\ref{app:ass}. Section~\ref{sec:comparison_between_erm_and_minmax_mom} presents the theoretical limits of the ERM compared to the minmax MOM estimators. Finally, Section~\ref{sec:simu} provides a simulation study where a natural algorithm associated to the  minmax MOM estimator for logistic loss is presented. The proofs of the main theorems are gathered in Sections ~\ref{app:proof},~\ref{lem:rad} and~\ref{proof_bernstein}. 
\paragraph*{Notations}
Let $\cX, \cY$ be measurable spaces and let $\bar{\cY}$ denote a convex set $\bar{\cY}\supset\cY$. Let $F$ be a class of measurable functions $f:\cX\to\bar{\cY}$ and let $(X,Y)\in\cX\times\cY$ be a random variable with distribution $P$. Let $\mu$ denote the marginal distribution of $X$. For any probability measure $Q$ on $\cX\times\cY$, and any function $g\in L_1(Q)$, let $Qg=\int g(x,y) \mathrm{d}Q(x,y)$.
Let $\ell : F\times\cX\times \cY\to \R,\ (f,x,y)\mapsto\ell_f(x,y)$ denote a loss function measuring the error made when predicting $y$ by $f(x)$. It is always assumed that there exists a function $\overline{\ell}:\bar{\cY}\times \cY\to \R$ such that, for any $(f,x,y)\in F\times\cX\times \cY$, $\overline{\ell}(f(x),y)=\ell_f(x,y)$. 
Let $R(f) = P \ell_f = \mathbb{E} \ell_f(X,Y)$ for $f$ in $F$ denote the risk and let $\mathcal{L}_f= \ell_f- \ell_{f^{*}}$ denote the excess loss. 
If $F\subset L_1(P):=L_1$ and Assumption~\ref{assum:lip_conv} holds, an equivalent risk can be defined even if $Y\notin L_1$.
Actually, for any $f_0\in F$, $\ell_f-\ell_{f_0}\in L_1$ so one can define $R(f)=P(\ell_f-\ell_{f_0})$.
W.l.o.g. the set of risk minimizers is assumed to be reduced to a singleton $\argmin_{f\in F}R(f)=\{f^*\}$.
$f^*$ is called the oracle as $f^{*}(X)$ provides the prediction of $Y$ with minimal risk among functions in $F$. 
%
%
%
%
For any $f$ and $p>0$, let $\|f\|_{L_p} = (P|f|^{p})^{1/p} $, for any $r\geqslant 0$, let $rB_{L_2} = \{ f \in F: \|f\|_{L_2} \leq r \}$ and  $rS_{L_2} = \{ f \in F: \|f\|_{L_2} = r \}$. For any set $H$ for which it makes sense, $H + f^{*} = \{ h+f^{*} \mbox{ s.t } h \in H   \}$, $H - f^{*} = \{ h-f^{*} \mbox{ s.t } h \in H   \}$. For any real numbers $a, b$, we write $a \lesssim b$ when there exists a positive constant $c$ such that $a \leq cb$, when $a \lesssim b$
and $b \lesssim a$, we write $a \asymp b$.


\section{ERM in the sub-Gaussian framework} \label{sec_ERM}
This section studies the ERM, improving some results from \cite{pierre2017estimation}. In particular, the global Bernstein condition in \cite{pierre2017estimation} is relaxed into a local hypothesis following \cite{MR3526202}.
All along this section, data $(X_i,Y_i)_{i=1}^N$ are independent and identically distributed with common distribution $P$. The ERM is defined for $f\in F \to P_N \ell_f = (1/N)\sum_{i=1}^{N} \ell_f(X_i,Y_i)$ by   
\begin{equation} \label{def_erm}
\hat{f}^{ERM} = \argmin_{f \in F} P_N \ell_f \enspace.
\end{equation}


The results for the ERM are shown under a sub-Gaussian assumption on the class $F-F$ with respect to the distribution of $X$. This result is the benchmark for the following minmax MOM estimators. 
\begin{Definition}
	Let $B \geq 1$. $F$ is called $B$-sub-Gaussian (with respect to $X$) when for all $f\in F$ and all $\lambda >1$
	\begin{equation*}
	\mathbb E \exp( \lambda |f(X)|/ \|f\|_{L_2} ) \leq \exp(\lambda^2 B^2/2)\enspace.
	\end{equation*} 
\end{Definition}
\begin{Assumption}
	\label{ass:sub-gauss} The class $F-F$ is $B$-sub-Gaussian with respect to $X$, where $F - F = \{f_1-f_2: f_1,f_2\in F\}$.
\end{Assumption}
Under this sub-Gaussian assumption, statistical complexity can be measured via Gaussian mean-widths. 

\begin{Definition}\label{def:gauss_mean_width}
Let $H\subset L_2$. Let $(G_h)_{h\in H}$ be the canonical centered Gaussian process indexed by $H$ (in particular, the covariance structure of $(G_h)_{h\in H}$  is given by  $\left(\E (G_{h_1}- G_{h_2})^2\right)^{1/2} = \left(\E(h_1(X)-h_2(X))^2\right)^{1/2}$ for all $h_1,h_2\in H$). The \textbf{Gaussian mean-width} of $H$ is $w(H) = \E \sup_{h\in H} G_h$.
\end{Definition}\label{def:fixed_point_gauss}
The complexity parameter driving the performance of $\hat{f}^{ERM}$ is presented in the following definition. 
\begin{Definition}\label{def:function_r}
	The \textbf{complexity parameter} is defined as 
	\begin{equation*}
	r_2(\theta ) \geq \inf \{ r>0 : 32Lw((F-f^*)\cap rB_{L_2}) \leq \theta r^2\sqrt{N}  \}
	\end{equation*}
	where $L>0$ is the Lipschitz constant from Assumption~\ref{assum:lip_conv}.
\end{Definition}
Let $A>0$. In \cite{MR2240689}, the class $F$ is called $(1,A)$-Bernstein if, for all $f\in F$, $P \cL_f^2\leqslant AP\cL_f$. Under Assumption~\ref{assum:lip_conv}, $F$ is $(1,AL^2)$-Bernstein if the following stronger assumption is satisfied
\begin{equation}\label{eq:BernCond}
\|f-f^*\|^2_{L_2}\leqslant AP\cL_f\enspace.
\end{equation}
This stronger version was used, for example in \cite{pierre2017estimation} to study ERM. However, under Assumption~\ref{assum:lip_conv}, Eq~\eqref{eq:BernCond} implies that 
\[
\|f-f^*\|^2_{L_2}\leqslant AP\cL_f\leqslant AL\|f-f^*\|_{L_1}\leqslant AL\|f-f^*\|_{L_2}\enspace.
\]
Therefore, $\|f-f^*\|_{L_2}\leqslant AL$ for any $f\in F$. The class $F$ is bounded in $L^2$-norm, which is restrictive as this assumption is not verified by the class of linear functions for example.
To bypass this issue, the following condition is introduced.
\begin{Assumption}\label{assum:fast_rates} There exists a constant $A > 0$ such that, for all $f\in F$ satisfying $\norm{f-f^*}_{L_2}= r_2(1/(2A) )$, we have $\|f-f^{*}\|_{L_2}^2\leqslant AP\mathcal{L}_f $.
\end{Assumption}
In Assumption~\ref{assum:fast_rates}, Bernstein condition is granted in a $L_2$-sphere centered in $f^*$ only. 
Outside of this sphere, there is no restriction on the excess loss. From the previous remark, it is clear that we necessarily have $r_2(1/(2A) ) \leq AL$ (as long as there exists some $f\in F$ such that $\norm{f-f^*}_2\geq r_2(1/2A)$).
This relaxed assumption is satisfied for many Lipschitz-convex loss functions under moment assumptions and weak assumptions on the noise as it will be checked in Section~\ref{app:ass}. 
The following theorem is the main result of this section.

\begin{Theorem}\label{theo:erm}
	Grant Assumptions~\ref{assum:lip_conv},~\ref{assum:convex},~\ref{ass:sub-gauss} and~\ref{assum:fast_rates}, $\hat{f}^{ERM}$ defined in \eqref{def_erm} satisfies, with probability larger than 
	\begin{equation} \label{eq:proba}
			1- 2\exp\left(-C\frac{Nr_2^2(1/(2A))}{(AL)^2}\right),
	\end{equation}
	\begin{equation} \label{eq:oracle}
	\|\hat{f}^{ERM}-f^{*}\|_{L_2 }^2 \leq  r_2^2(1/(2A))\mbox{ and }
	P \cL_{\hat{f}^{ERM}}\leq  \frac{r_2^2(1/(2A))}{2A}\enspace,
	\end{equation}
	where $C$ is an absolute constant.
\end{Theorem}
Theorem~\ref{theo:erm} is proved in Section~\ref{proof_erm}.
It shows deviation bounds both in $L_2$ norm and for the excess risk, which are both minimax optimal as proved in \cite{pierre2017estimation}.
As in \cite{pierre2017estimation}, a similar result can be derived if the sub-Gaussian  Assumption~\ref{ass:sub-gauss} is replaced by a boundedness in $L_\infty$ assumption. 
An extension of Theorem~\ref{theo:erm} can be shown, where Assumption~\ref{assum:fast_rates} is replaced by the following hypothesis: there exists $\kappa$ such that for all $f\in F$ in a $L_2$-shpere centered in $f^*$,  $\|f-f^{*}\|_{L_2}^{2\kappa}\leqslant AP\mathcal{L}_f $. 
The case $\kappa=1$ is the most classical and its analysis contains all the ingredients for the study of the general case with any parameter $\kappa\geq1$. More general Bernstein conditions can also be considered as in \cite[Chapter~7]{MR3526202}. These extensions are left to the interested reader.

Notice that none of the assumptions~\ref{assum:lip_conv},~\ref{assum:convex},~\ref{ass:sub-gauss} and~\ref{assum:fast_rates} involve the output $Y$ directly. 
All assumptions on $Y$ are done through the oracle $f^*$.
Yet, as will become transparent in the applications in Section~\ref{app:ass}, some assumptions on the distributions of $Y$ are required to check the assumptions of Theorem~\ref{theo:erm}. 
These assumptions are not very restrictive though and Lipschitz losses have been quite popular in robust statistics for this reason.

\section{Minmax MOM estimators}\label{sec:OUI}
This section presents and studies minmax MOM estimators, comparing them to ERM. We relax the sub-Gaussian assumption on the class $F-F$ and the i.i.d assumption on the data $(X_i,Y_i)_{i=1}^N$.

\subsection{The estimators}
The framework of this section is a relaxed version of the i.i.d. setup considered in Section~\ref{sec_ERM}. 
Following \cite{lecue2017learning,lecue2017robust}, there exists a partition $\cO\cup\cI$ of $\{1,\cdots,N\}$ in two subsets unknown to the statistician. 
No assumption is granted on the set of ``outliers" $(X_i,Y_i)_{i\in \cO}$. 
``Inliers", $(X_i,Y_i)_{i \in \mathcal{I}}$, are only assumed to satisfy the following assumption. 
For all $i\in\cI$, $(X_i,Y_i)$ has distribution $P_i$, $X_i$ has distribution $\mu_i$ and for any $p>0$ and any function $g$ for which it makes sense $\|g\|_{L_p(\mu_i)}=(P_i|g|^p)^{1/p}$.

\begin{Assumption}\label{assum:moments}
$(X_i,Y_i)_{i \in \mathcal{I}}$ are independent and, for any $i\in \cI$,
  $\|f-f^{*}\|_{L_2} = \|f-f^{*}\|_{L_2(\mu_i)} $ and $ P_i\mathcal{L}_f  = P\mathcal{L}_f $.
\end{Assumption}

Assumption~\ref{assum:moments} holds in the i.i.d case but it covers other situations where informative data $(X_i,Y_i)_{i \in \mathcal{I}}$ may have different distributions.
Typically, when $F$ is the class of linear functions on $\R^d$, $F=\{\inr{t,\cdot}, t\in \R^d\}$ and $(X_i)_{i\in \cI}$ are vectors with independent coordinates $(X_{i,j})_{j=1,\ldots,d}$, then Assumption~\ref{assum:moments} is met if the coordinates $(X_{i,j})_{i\in \cI}$ have the same first and second moments for all $j=1,\ldots,d$. \\
%
Recall the definition of MOM estimators of univariate means. 
Let $(B_k)_{k=1,\ldots,K}$ denote a partition of $\{1,\ldots,N\}$ into blocks $B_k$ of equal size $N/K$ (if $N$ is not a multiple of $K$, just remove some data). 
%
For any function $f:\cX\times\cY\to\R$ and $k\in \{1,\ldots,K\}$, let $P_{B_k} f = (K/N)\sum_{i \in B_k} f(X_i,Y_i)$.
MOM estimator is the median of these empirical means:
\[
  \MOM{K}{f} = \text{Med}(P_{B_1} f,\cdots,P_{B_K} f)\enspace.
\]
The estimator $\MOM{K}{f}$ achieves rate optimal sub-Gaussian deviation bounds, assuming only that $Pf^2<\infty$, see for example \cite{devroye2016sub}.
The number $K$ is a tuning parameter. The larger $K$, the more outliers are allowed. When $K=1$, $\MOM{K}{f}$ is the empirical mean, when $K=N$, the empirical median. 


Following \cite{lecue2017robust}, remark that the oracle is also solution of the following minmax problem:
\begin{align*}
f^{*} \in \argmin_{f \in F} P\ell_f = \argmin_{f \in F}  \sup_{g \in F} P(\ell_f-\ell_g)\enspace.
\end{align*}
Minmax MOM estimators are obtained by plugging MOM estimators of the unknown expectations $P(\ell_f-\ell_g)$ in this formula:
\begin{equation}\label{def:MinmaxMOM}
\hat{f} \in \argmin_{f \in F} \sup_{g \in F} \MOM{K}{\ell_f-\ell_g}\enspace.
\end{equation}

The minmax MOM construction can be applied systematically as an alternative to ERM. For instance, it yields a robust version of logistic classifiers. The minmax MOM estimator with $K=1$ is the ERM. 

The linearity of the empirical process $P_N$ is important to use localisation technics and derive ``fast rates'' of convergence for ERM \cite{MR2829871}, improving ``slow rates" derived with the approach of \cite{vapnik1998statistical}, see \cite{MR2051002} for details on ``fast and slow rates''. The idea of the minmax reformulation comes from \cite{MR2906886}, where this strategy allows to overcome the lack of linearity of some alternative robust mean estimators. \cite{lecue2017learning} introduced minmax MOM estimators to least-squares regression.





\subsection{Theoretical results}\label{sec:MainResults}
\subsubsection{Setting}
The assumptions required for the study of estimator \eqref{def:MinmaxMOM} are essentially those of Section~\ref{sec_ERM} except for Assumption~\ref{ass:sub-gauss} which is relaxed into Assumption~\ref{assum:moments}. Instead of Gaussian mean width, the complexity parameter is expressed as a fixed point of local Rademacher complexities \cite{bartlett2005local,MR2182250,MR2166554}. Let $(\sigma_i)_{i=1,\ldots, N}$ denote i.i.d. Rademacher random variables (uniformly distributed on $\{-1, 1\}$), independent from $(X_i,Y_i)_{i \in \cI}$. 
Let
\begin{equation} \label{comp:par}
\tilde{r}_2(\gamma) \geq \inf \bigg\{ r > 0, \forall J \subset \mathcal{I}: |J| \geq \frac{N}{2}, \mathbb{E}\sup_{f \in F: \|f-f^{*}\|_{L_2 } \leq r} \bigg | {\sum_{i \in J} \sigma_i(f-f^{*})(X_i)}   \bigg | \leq r^2 |J| \gamma \bigg\}\enspace.
\end{equation}

The outputs do not appear in the complexity parameter. This is an interesting feature of Lipschitz losses. It is necessary to adapt the Bernstein assumption to this framework.
\begin{Assumption}\label{assum:fast_rates_MOM} There exists a constant $A > 0$ such that for all $f\in F$ if $\|f-f^{*}\|_{L_2 }^2 = C_{K,r}$ then $\|f-f^{*}\|_{L_2 }^2\leqslant AP\mathcal{L}_f $ where
	\begin{equation}\label{eq:CKr}
	C_{K,r} =  \max\left( \tilde{r}_2^2(1/(575AL)), 864A^2L^2\frac{K}{N}\right) \enspace.
	\end{equation}
\end{Assumption}
Assumptions~\ref{assum:fast_rates_MOM} and~\ref{assum:fast_rates} have a similar flavor as both require the Bernstein condition in a $L_2$-sphere centered in $f^*$ with radius given by the rate of convergence of the associated estimator (see Theorems~\ref{theo:erm} and~\ref{theo:main}).
	For $K \leq (\tilde{r}_2^2(575/(AL))N)/(846A^2L^2)$ the sphere $\{f \in F : ||f-f^{*}||_{L_2} = \sqrt{C_{K,r}} \}$ is a $L_2$-sphere centered in $f^*$ of radius $\tilde{r}_2(575/(AL))$ which can be of order $1/\sqrt{N}$ (see Section~\ref{sec:example_basic}). 
	As a consequence, Assumption~\ref{assum:fast_rates_MOM} holds in examples where the small ball assumption does not (see discussion after Assumption~\ref{ass:L4L2}).

\subsubsection{Main results}
We are now in position to state the main result regarding the statistical properties of estimator~\eqref{def:MinmaxMOM} under a local Bernstein condition. 
\begin{Theorem}\label{theo:main}
    Grant Assumptions~\ref{assum:lip_conv},~\ref{assum:convex},~\ref{assum:moments} and~\ref{assum:fast_rates_MOM} and assume that $|\cO|\leq 3N/7$. Let $\gamma = 1/(575AL) $ and $K \in \big[  7|\cO|/3 ,  N \big]$. The minmax MOM estimator $\hat{f}$ defined in \eqref{def:MinmaxMOM} satisfies, with probability at least
    \begin{equation} \label{eq:proba_MOM}
 			1- \exp(-K/2016),
    \end{equation} 
\begin{equation} \label{eq:oracle}
\|\hat{f}-f^{*}\|_{L_2 }^2 \leq   C_{K,r}  \mbox{ and }
P \cL_{\hat f}\leq \frac{2}{3A} C_{K,r}  \enspace.
\end{equation}
\end{Theorem}
Suppose that $K = \tilde{r}_2^2(\gamma)N $, which is possible as long as $|\cO|\lesssim N \tilde{r}_2^2(\gamma)$. The deviation bound is then of order $\tilde{r}_2^2(\gamma)$ and the probability estimate $1-\exp(-N \tilde{r}_2^2(\gamma)/2016)$. Therefore, minmax MOM estimators achieve the same statistical bound with the same deviation as the ERM as long as $\tilde{r}_2^2(\gamma)$ and $r_2(\theta)$ are of the same order. 
Using generic chaining \cite{MR3184689}, this comparison is true under Assumption~\ref{ass:sub-gauss}. 
It can also be shown under weaker moment assumption, see \cite{MR3645129} or the example of Section~\ref{sec:example_basic}.
 
When $\tilde{r}_2^2(\gamma)\asymp r_2(\theta)$, the bounds are rate optimal as shown in \cite{pierre2017estimation}. 
This is why these bounds are called rate optimal sub-Gaussian deviation bounds. 
While these hold for ERM in the i.i.d. setup with sub-Gaussian design in the absence of outliers (see Theorem~\ref{theo:erm}), they hold for minmax MOM estimators in a setup where inliers may not be i.i.d., nor have sub-Gaussian design and up to $N \tilde{r}_2^2(\gamma)$ outliers may have contaminated the dataset.

%
This section is concluded by presenting an estimator achieving \eqref{def:MinmaxMOM} simultaneously for all $K$.
 For all $K\in\{1, \ldots, N\}$ and $f\in F$, define $T_K(f)=\sup_{g\in F} \MOM{K}{\ell_f-\ell_g} $ and let 
\begin{equation}\label{eq:confidence_sets}
\hat R_K = \left\{g\in F: T_K(g)\leq (1/3A) C_{K,r}\right\}\enspace.
\end{equation}
Now, building on the Lepskii's method, define a data-driven number of blocks 
\begin{equation}\label{eq:data_number_blocks}
\hat K  = \inf\left(K\in\{1, \ldots, N\}: \bigcap_{J=K}^N \hat R_J \neq \emptyset\right)
\end{equation}
and let $\tilde f$ be such that
\begin{equation}\label{eq:adapt_esti_K}
\tilde f \in \bigcap_{J=\hat K}^N \hat R_J\enspace.
\end{equation}

\begin{Theorem}\label{theo:MOM_lepski}
    Grant Assumptions~\ref{assum:lip_conv},~\ref{assum:convex},~\ref{assum:moments} and~\ref{assum:fast_rates_MOM} and assume that $|\cO|\leq 3N/7$. Let $\gamma = 1/(575AL) $. The estimator $\tilde{f}$ defined in \eqref{eq:adapt_esti_K} is such that for all $K \in \big[ 7|\cO|/3 ,  N \big]$, with probability at least
    \begin{equation*}
 			1- 4\exp(-K/2016),
    \end{equation*} 
\begin{equation*} 
\|\tilde{f}-f^{*}\|_{L_2 }^2 \leq   C_{K,r}  \mbox{ and }
P \cL_{\tilde f}\leq \frac{2}{3A} C_{K,r}  \enspace.
\end{equation*}
\end{Theorem}

Theorem~\ref{theo:MOM_lepski} states that $\tilde{f}$ achieves the results of Theorem~\ref{theo:main} simultaneously for all $K\geq 7 |\cO|/3$. 
This extension is useful as the number $|\cO|$ of outliers is typically unknown in practice. 
However, contrary to $\hat{f}$, the estimator $\tilde{f}$ requires the knowledge of $A$ and $\tilde{r}(\gamma)$.
These parameters allow to build confidence regions for $f^*$, which is necessary to apply Lepski's method.
Similar limitations appear in least-squares regression \cite{lecue2017robust} and even in the basic problem of univariate mean estimation. 
In this simpler problem, it can be shown that one can build sub-Gaussian estimators depending on the confidence level (through $K$) under only a second moment assumption.
On the other hand, to build estimators achieving the same risk bounds simultaneously for all $K$, more informations on the distribution of the data are required, see \cite[Theorem 3.2]{devroye2016sub}. 
In particular, the knowledge of the variance, which allows to build confidence intervals for the unknown univariate mean, is sufficient.
The necessity of extra-information to obtain adaptivity with respect to $K$ is therefore not surprising here.

\subsubsection{Some basic examples}\label{sec:example_basic}
The following example illustrates the optimality of the rates provided in Theorem~\ref{theo:main} even under a simple $L_2$-moment assumption.

\begin{Lemma}[\cite{MR2329442}]\label{Lem:LinReg}
In the $\mathcal{O} \cup  \mathcal{I}$ framework with $F = \{ \inr{t,\cdot} : t \in \mathbb{R}^d  \}$, we have $\tilde{r}_2^2(\gamma) \leq  \text{Rank}(\Sigma)/(2\gamma^2N)$, where $\Sigma=\mathbb{E} [XX^T]$ is the $d \times d$ covariance matrix of $X$.
\end{Lemma}
The proof of Lemma~\ref{Lem:LinReg} is recalled in Section \ref{lem:rad} for the sake of completeness. Lemma~\ref{Lem:LinReg} grants only the existence of a second moment for $X$ even though the rate obtained $\text{Rank}(\Sigma)/(2\gamma^2N)$ is the same as the one we would get under a sub-Gaussian assumption given that $r_2(\theta)\sim \text{Rank}(\Sigma)/(2\theta^2N)$. Moreover, Section \ref{app:ass} shows that Assumptions~\ref{assum:fast_rates} and~\ref{assum:fast_rates_MOM} are satisfied when $F= \{ \inr{t,\cdot} : t \in \mathbb{R}^d  \}$ and $X$ is a vector with i.i.d. entries having only a few finite moments. Theorem \ref{theo:main} applies therefore in this setting and the Minmax MOM estimator \eqref{def:MinmaxMOM} achieves the optimal fast rate of convergence $ \text{Rank}(\Sigma)/N$. This shows that when the model is the entire space $\bR^d$, the results for the ERM from Theorem~\ref{theo:erm} obtained under a sub-Gaussian assumption is the same as the one for the minmax MOM from Theorem~\ref{theo:main} under only weak moment assumption.

However, Lemma~\ref{Lem:LinReg} does not describe a typical situation. Having  $\tilde{r}_2^2(\gamma)$ of the same order as $r_2^2(\theta)$ under only a second moment assumption is mainly happening on large models such as the entire space $\bR^d$. For smaller size models such as the $B_1^d$-ball (the unit ball of the $\ell_1^d$-norm), the picture is different: $\tilde{r}_2^2(\gamma)$ should be bigger than $r_2^2(\theta)$ unless $X$ has enough moment. To make this statement simple let us consider the case $N=1$. In that case, we have $r_2^2(\theta) \sim \sqrt{\log d}$. Let us now describe $\tilde{r}_2^2(\gamma)$ under various moment assumptions on $X$ to see when $\tilde r_2^2(\gamma)$ compares with $\sqrt{\log d}$. 

Let $X=(x_j)_{j=1}^d$ be a random vector. It follows from Equation~(3.1) in \cite{MR2373017} that
\begin{equation*}
 \E \sup_{t\in B_1^d\cap r B_2^d} |\inr{t, X}| \lesssim
\left\{
\begin{array}{cc}
r\bE\left(\sum_{j=1}^d x_j^2\right)^{1/2} & \mbox{ if } r\leq 1/\sqrt{d}\\
 r \bE\left(\sum_{j=1}^{k}(x_j^*)^2\right)^{1/2} & \mbox{ if } 1/\sqrt{d}\leq r\leq 1\\
 \bE\max_{j=1, \ldots, d} |x_j| & \mbox{ if } r\geq1.
\end{array}\right.
\end{equation*}where $k = \lceil 1/r^2\rceil$  and $x_1^*\geq \ldots\geq x_d^*$ is a non-increasing rearrangement of the absolute values of the coordinates $x_j,j=1,\ldots,d$ of $X$. 
Assume that $x_1, \ldots, x_d$ are i.i.d. distributed like $x$ such that $\E x=0$ and $\E x^2=1$. Assume that $x$ has $2p$ moments, for $p \geq 1$ and let $c_0$ be such that  $\bE[x^{2p}] \leq c_0$. Then, using Jensen's inequality, we obtain
\begin{align*}
\left(\bE \left[\frac{1}{k}\sum_{j=1}^k (x_j^*)^2\right] \right)^{p}   \leq  \frac{1}{k}\sum_{j=1}^k \bE [(x_j^*)^{2p}]\leq \frac{1}{k}\sum_{j=1}^d \bE [x_j^{2p}] \leq \frac{c_0 d }{k}.
\end{align*}
It follows that 
\begin{align*}
\bE\left[ \left( \frac{1}{k}\sum_{j=1}^k (x_j^*)^2 \right)^{1/2} \right] \leq \left(\bE\left[ \frac{1}{k}\sum_{j=1}^k (x_j^*)^2\right]\right)^{1/2}\leq  \left(\frac{c_0d }{k}  \right)^{1/(2p)}   \leq  \left(\frac{c_0d}{k}\right)^{1/(2p)}\leq (c_0 d r^2)^{1/(2p)}\enspace.
\end{align*}
Hence, 
\begin{equation*}
 \E \sup_{t\in B_1^d\cap r B_2^d} |\inr{t, X}| \lesssim 
\left\{
\begin{array}{cc}
r\sqrt{d}& \mbox{ if } r\leq 1/\sqrt{d}\\
 (c_0 d r^2)^{1/(2p)} & \mbox{ if } 1/\sqrt{d}\leq r\leq 1\\
 (c_0 d)^{1/(2p)} & \mbox{ if } r\geq1.
\end{array}\right.\enspace.
\end{equation*}
Assume that $f^*=\inr{t^*,\cdot}$, with $t^*=0$.
Then,
\begin{equation*}
\tilde r_2(\gamma) = \inf\left\{r>0: \E \sup_{t\in B_1^d\cap r B_2^d} |\inr{t, X}|\leq \gamma r^2\right\}\lesssim (c_0 d)^{1/(4p)}/\sqrt{\gamma}\enspace.
\end{equation*}
In particular, 
\[
\begin{cases}
 \tilde r_2(\gamma)\asymp 1 &\text{ when }p\geq \log(c_0d),\\
1\lesssim r_2(\gamma)\lesssim \sqrt{\log d} &\text{ when }\log(c_0d)/\log\log d\leq p \leq \log(c_0d),\\
\tilde r_2(\gamma)\asymp d^{1/(4p)} &\text{ when }p\leq \log(c_0d)/\log\log d.
\end{cases}
\]
Let us now show that these estimates are sharp by considering $x=\eps(1+R \eta)$ where $\eps$ is a Rademacher variable, $\eta$ is a Bernouilli variable (independent of $\eps$) with mean $\delta=1/d$  and $R=d^{1/(4p)}$. We have $\bE x=0$ and $\bE x^2=1+R\delta\leq 2$ because $R\delta\leq1$ when $p\geq1$. Let $x_j=\eps_j(1+R \eta_j), j=1, \ldots, d$ be i.i.d. copies of $x$. We have 
\begin{align*}
  &\bE\max_{j=1, \ldots, d} |x_j|\geq (1+R) \bP\left[\max_{j=1, \ldots, d} |x_j|\geq 1+R\right]\\
  & = (1+R)\left(1-\bP[\eta_j=0,\forall j=1, \ldots, d]\right) = (1+R)(1-(1-\delta)^d)\geq (1+R)e^{-1}\gtrsim d^{1/(4p)}.
   \end{align*}As a consequence, for all $1\leq r\lesssim d^{1/(8p)}$, 
\begin{equation*}
 \E \sup_{t\in B_1^d\cap r B_2^d} |\inr{t, X}| = \bE\max_{j=1, \ldots, d} |x_j|> r^2
\end{equation*}and so $\tilde r_2^2(\gamma)\gtrsim d^{1/(4p)}$. As a consequence, under only a $L_{2p}$ moment assumption one cannot have $\tilde r_2^2(\gamma)$ better than $d^{1/(4p)}$.

As a consequence,  $\tilde r_2(\gamma)$ can be much larger than $r_2(\theta)$ when $x$ has less than $\log(c_0d)/\log(\log d)$ moments, for instance, $\tilde r_2^2(\gamma)$ can be of the order of $d^{1/8}$ when $x$ has only $2$ moments. This picture is different from the one given by Lemma~\ref{Lem:LinReg} where we were able to get equivalence between $\tilde r_2(\gamma)$ and $r_2(\theta)$ only under a second moment assumption.

\section{Relaxing the Bernstein condition} 
\label{sec:learning_without_assumption}
This section shows that minmax MOM estimators satisfy sharp oracle inequalities with exponentially large deviation under minimal stochastic assumptions insuring the existence of all objects. 
These results are slightly weaker than those of the previous section: the $L_2$ risk is not controlled and only slow rates of convergence hold in this relaxed setting.
However, the bounds are sufficiently precise to imply fast rates of convergence for the excess risk as in Theorems~\ref{theo:main}  if a slightly stronger Bernstein condition holds. 

Given that data may not have the same distribution as $(X,Y)$, the following relaxed version of Assumption~\ref{assum:moments} is introduced.
%

\begin{Assumption}\label{assum:moments2}
$(X_i,Y_i)_{i \in \mathcal{I}}$ are independent and for all $i\in\cI$, $(X_i,Y_i)$ has distribution $P_i$, $X_i$ has distribution $\mu_i$. For any $i\in \cI$, $F\subset L_2(\mu_i)$  and $ P_i\mathcal{L}_f  = P\mathcal{L}_f$ for all $f\in F$.
\end{Assumption}

When Assumption~\ref{assum:fast_rates_MOM} does not necessary hold, the localization argument has to be modified. 
Instead of the $L_2$-norm, the excess risk $f\in F \to P\cL_f$ is used to define neighborhoods around $f^*$. 
The associated complexity is then defined for all $\gamma>0$ and $K\in \{1,\cdots, N  \}$ by
\begin{equation} \label{eq:comp_par_local_excess_loss}
\bar{r}_2(\gamma)  \geq \inf \bigg\{ r > 0: \max\Big(\frac{E(r)}{\gamma}, \sqrt{1536}V_K(r)\Big)\leq r^2 \bigg\}
\end{equation}
where 
\begin{gather*}
E(r) = \sup_{J \subset \mathcal{I}:|J|\geq N/2}\mathbb{E}\sup_{f \in F: P\cL_f \leq r^2} \bigg | \frac{1}{ |J|}\sum_{i \in J} \sigma_i(f-f^{*})(X_i)   \bigg |\enspace,\\
\mbox{ and } V_K(r) =\max_{i\in\cI}\sup_{f \in F: P\cL_f \leq r^2}\sqrt{\mathbb{V}ar_{P_i}(\cL_f)} \sqrt{\frac{K}{N}}\enspace.
\end{gather*} 
There are two important differences between $\bar{r}_2(\gamma)$ on one side and  $r_2(\theta)$ in Definition~\ref{def:fixed_point_gauss} or $\tilde{r}_2(\gamma)$ in \eqref{comp:par} on the other side.
The first one is the extra variance term $V_K(r)$. 
Under the Bernstein condition, this term is negligible in front of the ``expectation term'' $E(r)$ see \cite{MR2240689}. 
In the general setting considered here, the variance term is handled in the complexity parameter. 
The second important consequence is that $\bar{r}_2$ is a fixed point of the complexity of $F$ localized around $f^*$ \emph{with respect to the excess risk} rather than with respect to the $L_2$-norm.
An important consequence is that this quantity is harder to compute in practical examples.
As a consequence, the results of this section are more of theoretical importance.

\begin{Theorem}\label{theo:main_without_bernstein_cond}
    Grant Assumptions~\ref{assum:lip_conv},~\ref{assum:convex},~\ref{assum:moments2} and assume that $|\cO|\leq 3N/7$. Let $\gamma=1/(768L)$ and  $K \in \big[  7|\cO|/3 ,  N \big]$. The minmax MOM estimator $\hat{f}$ defined in \eqref{def:MinmaxMOM} satisfies, with probability at least $1- \exp(-K/2016)$, $P\cL_{\hat f}\leq \bar{r}_2^2(\gamma).$
\end{Theorem}

Recall that Assumptions~\ref{assum:lip_conv} and \ref{assum:convex} are only meaning that the loss function is convex and Lipschitz and that the class $F$ is convex. 
Assumption~\ref{assum:moments2} says that inliers are independent and define the same excess risk as $(X,Y)$ over $F$. 
In particular, Theorem~\ref{theo:main_without_bernstein_cond} holds, as Theorem~\ref{theo:main}, without assumptions on the outliers $(X_i,Y_i)_{i\in\cO}$ and with weak assumptions on the outputs $(Y_i)_{i\in\cI}$ of the inliers (we remrak that excess loss function  $f\to P \cL_f$ is well-defined under no assumption on $Y$ -- even if $Y\notin L_1$ -- because $|\cL_f|\leq L|f-f^*|$). Moreover, the excess risk bound holds with exponentially large probability without assuming sub-Gaussian design, a small ball hypothesis or a Bernstein condition. 
This generality can be achieved by combining MOM estimators with convex-Lipschitz loss functions. 

The following result discuss relationships between Theorems~\ref{theo:main} and~\ref{theo:main_without_bernstein_cond}.
Introduce the following modification of the Bernstein condition. 
\begin{Assumption}\label{ass:bernstein_kappa} Let $\gamma = 1/(768L)$. There exists a constant $A > 0$ such that for all $f\in F$ if $P\cL_f = C^\prime_{K,r}$ then $\|f-f^{*}\|_{L_2 }^2 \leqslant AP\mathcal{L}_f $ where, for $\tilde{r}_{2}(\gamma)$ defined in~\eqref{comp:par},
	\begin{equation*}
	C^\prime_{K,r} =  \max\left( \frac{\tilde{r}_{2}^{2}(\gamma/A)}{A},  \frac{1536AL^2K}{N} \right)\enspace.
	\end{equation*}
\end{Assumption}
Assumption~\ref{ass:bernstein_kappa} is slightly stronger than Assumption~\ref{assum:fast_rates} since the $L_2$-metric to define the sphere is replaced by the excess risk metric. If Assumption~\ref{ass:bernstein_kappa} holds then Theorem~\ref{theo:bernstein_kappa} implies the same statistical bounds for \eqref{def:MinmaxMOM} as Theorem~\ref{theo:main} up to constants, as shown by the following result.

%


\begin{Theorem}\label{theo:bernstein_kappa}
Grant Assumptions~\ref{assum:lip_conv},~\ref{assum:convex},~\ref{assum:moments2} and assume that $|\cO|\leq 3N/7$. Assume that the local Bernstein condition Assumption~\ref{ass:bernstein_kappa} holds. Let $\gamma=1/(768L)$ and  $K \in \big[  7|\cO|/3 ,  N \big]$. The minmax MOM estimator $\hat{f}$ defined in \eqref{def:MinmaxMOM} satisfies, with probability at least $1- \exp(-K/2016)$, 
\begin{equation*}
\norm{\hat f-f^*}_{L_2}^2\leq \max\left(\tilde{r}_2^2(\gamma/A), \frac{1536L^2A^2K}{N}\right) \mbox{ and }P\cL_{\hat f}\leq  \max\left(\frac{\tilde{r}_2^2(\gamma/A)}{A}, \frac{1536L^2AK}{N}\right).
\end{equation*}
\end{Theorem}
\begin{proof}
 First, $V_K(r)\leq LV_K^\prime(r)$ for all $r>0$ where $V_K^\prime(r) =\sqrt{K/N} \max_{i\in\cI}\sup_{f \in F: P\cL_f \leq r^2}\norm{f-f^*}_{L_2(\mu_i)}$. Moreover, $r\to E(r)/r^2$ and $r\to V_K^\prime(r)/r^2$ are non-increasing, therefore by
Assumption~\ref{ass:bernstein_kappa} and the definition of $\tilde{r}_2(\gamma)$, $V^\prime_K(r)$,
 \begin{equation*}
 	\frac{1}{\gamma} E \bigg( \frac{\tilde{r}_{2}(\gamma/A)}{\sqrt{A}} \bigg) \leq \frac{\tilde{r}_{2}^{2}(\gamma/A)}{A} \quad \mbox{and} \quad \sqrt{1536} V_K^\prime\bigg( \sqrt{1536} L\sqrt{\frac{AK}{N}} \bigg) \leq \frac{1536 A^2LK}{N}\enspace.
\end{equation*}
Hence, $\bar{r}_2^2(\gamma) \leq \max(\tilde{r}_2^2(\gamma/A)/A, 1536L^2A(K/N))$. 
\end{proof}


\section{Bernstein's assumption }\label{app:ass}

This section shows that the local Bernstein condition holds for various loss functions and design $X$. In Assumption~\ref{assum:fast_rates} and~\ref{assum:fast_rates_MOM}, the comparizon between $P\cL_f$ and $\|f-f^*\|_{L_2}^2$ is only required on a $L_2$-sphere. In this section, we prove that the local Bernstein assumption can be verified over the entire $L_2$-ball and not only on the sphere under mild moment conditions. 
The class $F-\{f^*\}$ is assumed to satisfy a `` $L_{2+\varepsilon}/L_2$-norm equivalence assumption'', for $\varepsilon >0 $.
\begin{Assumption} \label{ass:L4L2} Let $\varepsilon >0$.
	There exists $C'>0$ such that for all $ f\in F$, $\|f-f^{*}\|_{L_{2+\varepsilon}} \leq C'\|f-f^{*}\|_{L_2}$ 
\end{Assumption}
Assumption~\ref{ass:L4L2} is a ``$L_{2+\varepsilon}/L_2$'' norm equivalence assumption over $F-\{f^*\}$. A ``$L_{4}/L_2$'' norm equivalence assumption over $F-\{f^*\}$ has been used for the study of MOM estimators (see \cite{LugosiMendelson2016}).  Examples of distributions satisfying Assumption~\ref{ass:L4L2}  can be found in \cite{mendelson2014learning,MR3367000}. 

There are situations where the constant $C'$ depends on the dimension $d$ of the model. 
In that case, the results in \cite{LugosiMendelson2016,lecue2017robust} provide sub-optimal statistical upper bounds.  
For instance, if $X$ is uniformly distributed on $[0,1]$ and  $F=\{\sum_{j=1}^d \alpha_j I_{A_j}:(\alpha_j)_{j=1}^p\in\R^d\}$ where $I_{A_j}$ is the indicator of $A_j=[(j-1)/d,j/d]$ then  for all $f\in F$, $\norm{f-f^*}_{L_{2+\varepsilon}}\leq d^{\varepsilon/(4+2\varepsilon)} \norm{f-f^*}_{L_2}$ so $C' = d^{\varepsilon/(4+2\varepsilon)} $. This dependence with respect to the dimension $d$ is inevitable. For instance, in  \cite{LugosiMendelson2016,lecue2017robust}, a $L_4/L_2$ norm equivalence is required. In this case, $C' = d^{1/4}$ which ultimately yields sub-optimal rates in this example. 
On the other hand, as will become clear in this section, the rates given in Theorem~\ref{theo:main} or Theorem~\ref{theo:MOM_lepski} are not deteriorated in this example. 
This improvement is possible since the Bernstein condition is only required in a neighborhood of $f^*$.

\subsection{Quantile loss}

The proof is based on \cite[Lemma 2.2]{elsener2016robust} and is postponed to Section~\ref{proof:thm4}. 
Recall that $\ell_f(x,y) = (y-f(x))(\tau - I \{  y-f(x) \leq 0 \})$.

\begin{Assumption} \label{ass:quantil} 
Let $C'$ be the constant defined in Assumption \ref{ass:L4L2}. There exist $\alpha>0$ and $r>0$ such that, for all $x\in\cX$ and for all $z$ in $\mathbb{R}$ such that $ |z-f^{*}(x) | \leq r (\sqrt{2}C')^{(2+\varepsilon)/\varepsilon} $, we have $ f_{Y|X=x}(z) \geq \alpha$, where $f_{Y|X=x}$ is the conditional density function of $Y$ given $X=x$.
\end{Assumption}

\begin{Theorem} \label{quantile_loss}
	Grant Assumptions \ref{ass:L4L2} (with constant $C^\prime$) and \ref{ass:quantil} (with parameter $r$ and $\alpha$). Then, for all $f\in F$ satisfying $\|f-f^*\|_{L_2}\leqslant r$, $\|f-f^*\|_{L_2}^2\leqslant (4/\alpha) P\mathcal{L}_f$.	
\end{Theorem}

Consider the example from Section~\ref{sec:example_basic}, assume that $K\lesssim {\rm Rank}(\Sigma)$ and let $r^2 = C_{K,r} = \tilde{r}_2^2(\gamma) \leq {\rm Rank}(\Sigma) / (2\gamma^2 N)$. 
If $C'=d^{\varepsilon /(4+2\varepsilon)}$, Assumption~\ref{ass:quantil} holds for $r$ and an associated $\alpha\gtrsim 1$ as long as $d^{1/2}\sqrt{{\rm Rank}(\Sigma) / N}\lesssim 1$ and, for all $x\in\cX$ and for all $z$ in $\mathbb{R}$ such that $ |z-f^{*}(x) | \lesssim 1 $, $ f_{Y|X=x}(z) \gtrsim 1$.  As ${\rm Rank}(\Sigma)\leqslant d$, the first condition reduces to $N\gtrsim d^2$. In this situation, the rates given in Theorems~\ref{theo:main} and~\ref{theo:MOM_lepski} are still ${\rm Rank}(\Sigma)/N$. This gives a partial answer, in our setting, to the issue raised in \cite{saumard2018optimality} regarding results based on the small ball method.


%

%
%
%
%

\subsection{Huber Loss}
Consider the Huber loss function defined, for all $f\in F$, $x\in\cX$ and $y\in\R$, by $\ell_f(x,y) = \rho_H(y-f(x))$ where $\rho_H(t) = t^2/2$ if $|t| \leq \delta$ and $\rho_H(t) = \delta |t|- \delta^2/2$ otherwise. 
Introduce the following assumption.
\begin{Assumption} \label{ass:huber}
		Let $C'$ be the constant defined in Assumption \ref{ass:L4L2}. 
		There exist $\alpha>0$ and $r>0$ such that for all $x\in\cX$ and all $z$ in $\mathbb{R}$ such that $ |z-f^{*}(x) | \leq (\sqrt{2}C')^{(2+\varepsilon)/\varepsilon}r$, $F_{Y|X=x}(z+\delta) - F_{Y|X=x}(z- \delta)\geqslant \alpha$, where $F_{Y|X=x}$ is the conditional cumulative function of $Y$ given $X=x$.
\end{Assumption}
Under this assumption and a ``$L_{2+\eps}/L_2$'' assumption, the local Bernstein condition is proved to be satisfied in the following result whose proof is postponed to Section~\ref{proof:thm5}. 
\begin{Theorem} \label{huber_loss}
	Grant Assumptions \ref{ass:L4L2} (with constant $C^\prime$) and \ref{ass:huber} (with parameter $r$ and $\alpha$). Then, for all $f\in F$ satisfying $\|f-f^*\|_{L_2}\leqslant r$, there exists $\alpha>0$ (given by Assumption~\ref{ass:huber}) such that $\|f-f^*\|_{L_2}^2\leqslant (4/\alpha) P\mathcal{L}_f$.
\end{Theorem}

\subsection{Logistic classification} \label{bern:log}
In this section we consider the logistic loss function. 

\begin{Assumption}\label{ass:log}
	There exists $c_0 > 0$ such that
	\begin{align*}
	\mathbb{P} \big( |f^*(X)| \leq c_0   \big) \geq 1-\frac{1}{(2C')^{(4+2\varepsilon)/\varepsilon}} .
	\end{align*} 
where $C'$ is defined in Assumption~\ref{ass:L4L2}. 
\end{Assumption} 
The following result is proved in Section~\ref{proof:thm7}.
\begin{Theorem}\label{thm:AppliLog}
	Grant Assumptions \ref{ass:L4L2} and \ref{ass:log}. Then, for all $r>0$  and all $f\in F$ such that $ \|f-f^*\|_{L_2}\leqslant r$,
	\[
	P\mathcal{L}_f\geqslant \frac{e^{ -c_0 - r(2C')^{(2+\epsilon)/\epsilon}   }  }{2\big( 1+ e^{c_0 + r(2C')^{(2+\epsilon)/\epsilon} }  \big)^2 } \|f-f^*\|_{L_2}^2 \enspace.
	\]
\end{Theorem}
The proof is postponed to Section~\ref{proof:thm7}. As for the Huber Loss and the Hinge Loss, the rates of convergence are not deteriorated when $C'$ may depend on the dimension as long as $r \times (C')^{(2+\varepsilon)/\varepsilon}$ is smaller than some absolute constant. 

\subsection{Hinge loss}

In this section, we show that the local Bernstein condition holds for various design $X$ for the Hinge loss function. We obtain the result under the assumption that the oracle $f^*$ is actually the Bayes rules which is the function minimizing the risk $f \mapsto R(f)$ over all measurable functions from $\cX$ to $\bR$. 
Recall that, under this assumption, $f^*(x) = \mbox{sign}(2\eta(x)-1)$ where $\eta(X) = \bP (Y=1|X)$. 
In that case, the Bernstein condition (see \cite{MR2240689}) coincides with the margin assumption (see \cite{MR2051002,MR1765618}).
\begin{Assumption}\label{ass:log}
	Let $C'$ be the constant defined in Assumption \ref{ass:L4L2}. There exist $\alpha > 0$ and $0<r \leq (\sqrt{2}C')^{-(2+\varepsilon)/\varepsilon}$ such that for all $x \in \cX$, for all $z \in \bR$, $|z-f^*(x)| \leq (\sqrt{2}C')^{(2+\varepsilon)/\varepsilon}r$
	\begin{align*}
		\min \big( \eta(x), 1- \eta(x) , |1-2\eta(x)|  \big) \geq \alpha \enspace.
	\end{align*} 
\end{Assumption} 
Assumption~\ref{ass:log} is also local and has the same flavor as Assumptions~\ref{ass:quantil} and~\ref{ass:huber}.

\begin{Theorem} \label{thm:hinge}
	Grant Assumptions \ref{ass:L4L2} (with constant $C^\prime$) and \ref{ass:log} (with parameter $r$ and $\alpha$). Assume that the oracle $f^*$ is the Bayes estimator i.e. $f^*(x) = \mbox{sign}(2\eta(x)-1)$ for all $x\in\cX$.  Then, for all $f\in F$ such that $\|f-f^*\|_{L_2}\leqslant r$, $\|f-f^*\|_{L_2}^2\leqslant \frac{2}\alpha P\mathcal{L}_f$.	
\end{Theorem}
The proof is postponed to Section~\ref{proof:thm8}. 

\section{Comparison between ERM and minmax MOM} 
\label{sec:comparison_between_erm_and_minmax_mom}
In this section, we show that robustness properties with respect to heavy-tailed data and to outliers of the minmax MOM estimator in Theorem~\ref{theo:main} cannot be achieved by the ERM. 
We prove two lower bounds on the statistical risk of ERM. 
First, we show that ERM is not robust to contamination in the design $X$ and second that ERM cannot achieve the optimal rate with a sub-Gaussian deviation under only moment assumptions. 

We first show the absence of robustness of ERM w.r.t. contamination by even a single input variable. 
We consider the absolute loss function of linear functionals $\ell_t(x,y)=|y-\inr{x,t}|$.
Let $X_1,\ldots,X_N$ denote i.i.d. Gaussian vectors, and suppose that there exists $t^*$ such that $Y_i=\inr{X_i, t^*}, i=1, \ldots, N$.
Assume that a vector $v\in\bR^d$ was added to $X_1$ (and that this is the only corrupted data). 
Hence, we are given the dataset $(X_1+v, Y_1), (X_2, Y_2), \cdots, (X_N, Y_N)$. 
Consider the ERM constructed on this dataset i.e $\hat t^{ERM} \in \argmin_{t\in\bR^d}P_N\ell_t$ where $P_N\ell_t = (1/N)|Y_1-\inr{X_1+v, t}| + (1/N)\sum_{i=2}^N |Y_i-\inr{X_i, t}|$. In this context, the following lower bound holds.

\begin{Proposition}\label{prop:lower_bound_ERM_contamination}
There exist absolute constants $c_0$ and $c_2$ such that the following holds. If the contamination vector $v$ satisfies $|\inr{v,t^*}|\geq (1/2)\norm{v}_2\norm{t^*}_2$, with $\norm{v}_2\geq c_2N$, then with probability at least $1-4\exp(-c_0 N)$, $\norm{\hat t^{ERM} - t^*}_2\geq \norm{t^*}_2/4$.
\end{Proposition}
When $N \asymp d$, from Theorem~\ref{theo:main} with $K \asymp d$, minmax MOM estimators yields, with probability at least $1-2\exp(-cd)$, $\norm{\hat t^{MOM} - t^*}_2\lesssim d/N$ on the same dataset as the one used by ERM in Proposition~\ref{prop:lower_bound_ERM_contamination}. If $\|t^*\| \gtrsim 1$ then the ERM is suboptimal compared with the minmax MOM estimator.

\beginproof
To show that $\hat t^{ERM}$ is outside $\cB=B_2(t^*, (1/4)\norm{t^*}_2)= \{t\in\bR^d: \norm{t-t^*}_2\leq (1/4)\norm{t^*}_2\}$, it is enough to show that $P_N\ell_0$ is smaller than the smallest value of $t\to P_N \ell_t$ over $\cB$. 
It follows from Gaussian concentration that, with probability at least $1-\exp(-c_0N)$,
\begin{equation}\label{eq:risk_zero}
P_N\ell_0 = \frac{1}{N}\sum_{i=1}^N |\inr{X_i,t^*}|\leq \frac{3}{2}\sqrt{\frac{2}{\pi}}\norm{t^*}_2.
\end{equation}

Let us now bound from bellow the empirical loss function $t\to P_N\ell_t$ uniformly over all $t$ in $\cB$. 
First,
\begin{equation}\label{eq:inr_v_t}
|\inr{v,t}|\geq |\inr{v,t^*}|-|\inr{v,t-t^*}|\geq \norm{v}_2(\norm{t^*}_2/2-\norm{t-t^*}_2)\geq \norm{v}_2\norm{t^*}_2/4.
\end{equation}Then, it follows from Borell-TIS inequality (see Theorem~7.1 in \cite{MR1849347} or pages 56-57 in \cite{ledoux2013probability}) that with probability at least $1-2\exp(-c_1N), \norm{X_1}_2\leq \bE \norm{X_1}_2 + \sqrt{2c_1N}\leq c_3\sqrt{N+d}$. Therefore, $\bP(\Omega_1)\geq 1-2\exp(-c_1N)$, where
\[
\Omega_1=\{\forall t\in\bR^d,\quad |\inr{X_1, t-t^*}|\leq c_3\sqrt{N+d}\norm{t-t^*}_2\}\enspace.
\]
On $\Omega_1$, we have for all $t\in \cB$ that  $|\inr{X_1,t^*-t}|\leq c_3\sqrt{N+d}\norm{t-t^*}_2\leq (c_3/4)\sqrt{N+d}\norm{t^*}_2$.
Therefore, using \eqref{eq:inr_v_t} and $\norm{v}_2\geq c_2N$ for a large enough constant $c_2$,
\begin{align}\label{eq:risk_ball_t_star}
\notag P_N \ell_t  & = \frac{1}{N}|\inr{X_1,t^*-t} - \inr{v,t}| + \frac{1}{N}\sum_{i=2}^N |\inr{X_i,t^*-t}|\geq \frac{1}{N}|\inr{v, t}| - \frac{1}{N}|\inr{X_1, t^*-t}|\\
&\geq \frac{1}{4N}\norm{v}_2\norm{t^*}_2 - \frac{c_3\norm{t^*}_2}{4\sqrt{N}}>\frac{3}{2}\sqrt{\frac{2}{\pi}}\norm{t^*}_2\enspace.
\end{align}
\endproof

It follows from Proposition~\ref{prop:lower_bound_ERM_contamination} that ERM is not consistant when there is even a single outlier among the $X_i$. 
By comparison, the minmax MOM has optimal performance even when the dataset has been corrupted by up to $d$ outliers when $N\gtrsim d$.
This shows a first advantage of the minmax MOM approach. 

Now, we prove a second advantage of the minmax MOM over the ERM by considering heavy-tailed design. 
We also consider the absolute $L_1$-loss function as in the previous example and suppose that data are generated from a linear model in dimension $d=1$: $Y=Xt^*+\zeta$ where $X$ and $\zeta$ are independent mean zero random variables and $t^*\in\bR$ (we choose $d=1$ so that we have access to a canonical definition of median which simplifies the proof). Our aim is to show that if the design  $X$ has only a second moment then the ERM $\hat t^{ERM}$ cannot achieve the optimal rate $\sqrt{x/N}$ with a sub-Gaussian deviation that is $1-\exp(-c_0x)$ as does the minmax MOM for all $x\in[1,N]$.   

\begin{Proposition}\label{prop:lower_bound_ERM_moment} Let $N\geq 8000$ and $10\leq x\leq N/800$. There exist $X$ and $\zeta$ two symmetric and independent random variables such that $\E X^2\in[1, 16]$, $\bE \zeta^2\leq 5 x^2$ and, for any $t^*\in\bR$ and $Y=Xt^*+\zeta$, we have $\{t^*\} = \argmin_{t\in\bR}\bE|Y-Xt|$. 
Let $(X_i,Y_i)_{i=1}^N$ be $N$ i.i.d. copies of $(X,Y)$ such that $Y=Xt^*+\zeta$ for some $t^*\in\bR$. 
Let $\hat t^{ERM}\in\argmin_{t\in\bR}\sum_{i=1}^N|Y_i-X_it|$. 
Then, with probability at least $3/(5x)$, 
\[
\sqrt{\bE[(X(\hat t^{ERM} - t^*))^2]}\geq (1/5)\sqrt{x/N}\enspace.
\]  
\end{Proposition}

\beginproof
Let $\delta^\prime = (1/8)\sqrt{x/(2N)}$ and let $\zeta$ be uniformly distributed over $[-x-1/2+\delta^\prime, -x]\cup[-\delta^\prime, \delta^\prime]\cup[x, x+1/2-\delta^\prime]$. 
Let $\eps$ denote a Rademacher variable, let $\eta$ be a Bernoulli variable with parameter $\delta=1/(xN)$ and $R=4/\sqrt{\delta}=4\sqrt{xN}$. 
We assume that $\zeta, \eps$ and $\eta$ are independent and let $X=\eps(1+R\eta)$.  
Let $t^*\in\bR$ and let $\cD_N=(X_i,Y_i)_{i=1}^N$ be a dataset of $N$ i.i.d. copies of $(X,Y)$, where $Y=Xt^*+\zeta$. 

Since the median of $\zeta$ is $0$, for all $u\in\bR, \bE|u-\zeta|\geq \bE |\zeta|$ with equality iff $u=0$. 
As a consequence, for all $t\in\bR, \bE|Y-Xt| = \bE_X\bE_\zeta|\zeta - X(t^*-t)|\geq \bE|\zeta|$ and the only minimizer of $t\in\bR\to\bE|Y-Xt|$ is $t^*$. 
In other words, $t^*$ is the oracle. 
For all $t\in\bR$, 
\[
\E [(X(t-t^*))^2]=\bE [X^2] (t-t^*)^2 = (1+2R\delta + R^2\delta)(t-t^*)^2\enspace.
\]
Since $R^2\delta=16$ and $2R\delta\leq \sqrt{8}/100\leq 1$, we have $(t-t^*)^2\leq \E (X(t-t^*))^2\leq 18(t-t^*)^2$, that is, the $L^2(\mu)$-norm is equivalent to the absolute value.

Observe that  $\hat t^{ERM} -t^*$ is solution of the minimization problem
\begin{equation*}
(\hat t^{ERM} -t^*)\in\argmin_{u\in\bR} \sum_{i=1}^N |X_i| \left|\frac{\zeta_i}{X_i} - u\right|=\argmin_{u\in\bR}\E[|W-u||\cD_N]\enspace.
\end{equation*} 
Here, defining $\zeta_i^\prime = \epsilon_i\zeta_i, i\in[N]$, $W$ is a random variable such that
\[
\bP\bigg[W = \frac{\zeta_i^\prime}{|X_i|}|\cD_N\bigg]=\frac{|X_i|}{ \sum_{i=1}^N |X_i|}\enspace.
\]
Notice that, almost surely, all $\zeta_i^\prime/|X_i|$ are different. 
In particular, $|\hat t^{ERM} -t^*|$ is the absolute value of the empirical median $|\med{W}|$.
Therefore, $|\hat t^{ERM} -t^*|\geq c_1\sqrt{x/N}$ when the median of $W$ does not belong to $(-c_1\sqrt{x/N}, c_1\sqrt{x/N})$.
This holds when $\bP[W\leq -c_1\sqrt{x/N}|\cD_N]>1/2$ or $\bP[W\geq c_1\sqrt{x/N}|\cD_N]>1/2$. 
Introduce the following sets
\begin{equation*}
I_{\leq -x}:=\left\{i\in[N]: \zeta_i^\prime\leq -x\right\}, I_{\delta^\prime}:=\left\{i\in[N]: |\zeta_i^\prime|\leq \delta^\prime\right\} \mbox{ and } I_{\geq x}:=\left\{i\in[N]: \zeta_i^\prime\geq x\right\} \enspace.
\end{equation*}
Define also the following events
\begin{gather*}
\Omega_0:=\left\{|I_{\delta^\prime}|\leq \sqrt{2xN} ,\qquad \big||I_{\leq -x}| - |I_{\geq x}|\big|\leq \sqrt{2xN}\right\} \enspace,\\  
\Omega_1:=\left\{\forall i\in I_\delta: \eta_i=0 \mbox{ and } |\{i\in[N]:\eta_i=1\}|=1\right\}\enspace. 
\end{gather*}
By Hoeffding's  inequality (see Chapter~2 in \cite{MR3185193}), as $(\zeta_i^\prime)_{i=1}^N$ is a family of i.i.d. random variables distributed like $\zeta_1$, with probability at least $1-\exp(-x/4)$,
\begin{equation*}
|I_{\delta^\prime}| = \sum_{i=1}^N I(|\zeta_i^\prime|\leq \delta^\prime)\leq N \bP[|\zeta_1^\prime|\leq \delta^\prime] + \sqrt{\frac{xN}{2}}=2\delta^\prime N + \sqrt{\frac{xN}{2}}\leq \sqrt{2xN}\enspace.
\end{equation*}
Since $\bP[\zeta_1^\prime\leq -x] = \bP[\zeta_1^\prime\geq x]$, $I(\zeta_i^\prime\leq -x) - I(\zeta_i^\prime\geq x)$ are independent, centered random variables taking values in $[-1,1]$. 
By Hoeffding's inequality, with probability at least $1-2\exp(-x/2)$,
\begin{equation*}
\big||I_{\leq -x}| - |I_{\geq x}|\big|=\left|\sum_{i=1}^N I(\zeta_i^\prime\leq -x) - I(\zeta_i^\prime\geq x)\right|\leq \sqrt{2xN}.
\end{equation*}
Using a union bound, we have $\bP[\Omega_0]\geq 1-2\exp(-x/2)-\exp(-x/4)\geq 1-1/(10x)$ when $x\geq10$. 
Since the $\zeta^\prime_i$'s and the $\eta_i$'s are independent, on the event $\Omega_0$, we have
\begin{equation*}
\bP[\forall i\in I_{\delta'}: \eta_i=0|\cD_N] = (1-\delta)^{|I_{\delta^\prime}|}\geq (1-\delta)^{\sqrt{2xN}}\geq 1-2\delta\sqrt{2xN} = 1-\frac{2\sqrt{2}}{\sqrt{xN}}\geq 1-\frac{1}{10x}\enspace.
\end{equation*}
The last inequality holds since $x\leq N/800$. 
Moreover, 
\begin{equation*}
\bP[|\{i\in[N]:\eta_i=1\}|=1] = N\delta (1-\delta)^{N-1}\geq \frac{1-2/x}{x}\enspace.
\end{equation*}
When $x\geq10$, this implies
\begin{equation*}
\bP[|\{i\in[N]:\eta_i=1\}|=1]\geq \frac{4}{5x}\enspace.
\]
Therefore, on the event $\Omega_0$,  $\bP[\Omega_1|\cD_N]\geq 7/(10x)$ and so 
\[
\bP[\Omega_0\cap \Omega_1]=\E[{\bf 1}_{\Omega_0}\E[{\bf 1}_{\Omega_1}|\cD_N]\geq \frac{4}{5x}\bigg(1-\frac1{10x}\bigg)\geq 3/(5x)\enspace.
\] 

We want to show that $|\med{W}|\geq c_1\sqrt{x/N}$ on the event $\Omega_0\cap \Omega_1$, for some well-chosen constant $c_1>0$. 
We have $\bP[W\leq -c_1\sqrt{x/N}|\cD_N]>1/2$ if and only if
\begin{equation*}
   \sum_{i=1}^N I\left(\frac{\zeta_i^\prime}{|X_i|}\leq -c_1\sqrt{x/N}\right)|X_i|\geq \frac{1}{2}\sum_{i=1}^N |X_i|\enspace.
   \end{equation*} 
   In particular, if
   \begin{equation}\label{eq:median_sets_I_x}
   \sum_{i\in I_{\leq -x}} |X_i|\geq \frac{1}{2}\sum_{i=1}^N |X_i| 
  \mbox{ or } \sum_{i\in I_{\geq x}} |X_i|\geq \frac{1}{2}\sum_{i=1}^N |X_i|
   \end{equation}
then the median of $W$ takes value in $\{\zeta_i^\prime/|X_i|:i\in I_{\leq -x}\}$, resp. in  $\{\zeta_i^\prime/|X_i|:i\in I_{\geq x}\}$. 
 Since, for all $i\in I_{\leq -x}, \zeta_i/|X_i|\leq -x/(1+R)<-\delta^\prime$ and for all $i\in I_{\geq x}, \zeta_i/|X_i|\geq x/(1+R)>\delta^\prime$, in these cases, $|\med{W}|\geq x/(1+R) = x/(1+4\sqrt{xN})\geq (1/5)\sqrt{x/N}$.
Since $|\hat t^{ERM} -t^*| = |\med{W}|$, the proof is finished if \eqref{eq:median_sets_I_x} is proved.

Let us now prove that \eqref{eq:median_sets_I_x} holds on the event $\Omega_0\cap \Omega_1$. 
On this event, only one $\eta_i$ equals to $1$.
Therefore only one $|X_i|$ equals to $1+R$ and all the others equal $1$. 
Moreover, $\eta_i=0$ for all $i\in I_{\delta^\prime}$. 
Therefore, if $i^*\in[N]$  denotes the only index such that $\eta_{i^*}=1$, then either $i^*\in I_{\leq -x}$ or $i^*\in I_{\geq x}$. 
If $i^*\in I_{\leq -x}$, on $\Omega_0\cap \Omega_1$,
\begin{equation*}
\sum_{i\in I_{\leq -x}}|X_i| = |I_{\leq x}| - 1 + (1+R) = |I_{\leq -x}| + 4\sqrt{xN}\geq |I_{\geq x}|-\sqrt{\frac{xN}{2}} + |I_{\delta^\prime}|-\sqrt{2xN} + 4\sqrt{xN}\geq |I_{\geq x}| + |I_{\delta^\prime}|\enspace.
\end{equation*}
Moreover, all the $|X_i|$ equal $1$ when $i\in I_{\geq x}\cup I_{\delta^\prime}$.
Therefore, $|I_{\geq x}| + |I_{\delta^\prime}| = \sum_{i\in I_{\geq x}\cup I_{\delta^\prime}} |X_i|$. 
Overall, $\sum_{i\in I_{\leq -x}}|X_i|\geq \sum_{i\in I_{\geq x}\cup I_{\delta^\prime}} |X_i|$ which is equivalent to $\sum_{i\in I_{\leq -x}} |X_i|\geq (1/2)\sum_{i=1}^N |X_i|$. 
Likewise, if $i^*\in I_{\geq x}$  then $\sum_{i\in I_{\geq x}} |X_i|\geq (1/2)\sum_{i=1}^N |X_i|$. 
Therefore, on the event $\Omega_0\cap \Omega_1$ \eqref{eq:median_sets_I_x} holds.
\endproof

Proposition~\ref{prop:lower_bound_ERM_moment} shows that the distance between the ERM and $t^*$ is larger than $(1/5)\sqrt{x/N}$ with probability at least $3/(5x)$. 
This probability is larger than $1-\exp(-x/2016)$ for large values of $x$, which shows that the ERM does not have sub-Gaussian deviations.
Let us now show that, using the same data and a number of blocks $K\asymp x$ (or using the adaptive estimator \eqref{eq:adapt_esti_K}), the minmax MOM estimators achieve the rate $\sqrt{x/N}$ with probability at least $1-\exp(-x/2016)$. 
This will show a second advantage of minmax MOM estimators compared with ERM for heavy-tailed designs.

To apply Theorem~\ref{theo:main} (or Theorem~\ref{theo:MOM_lepski} for the adaptive estimator), we show that the local Bernstein condition is satisfied for the example of Proposition~\ref{prop:lower_bound_ERM_moment} and compute the complexity parameter $\tilde r_2(\gamma)$. We have
\begin{align*}
&\bE \sup_{t\in\bR: \bE[(X(t-t^*))^2]\leq r^2}\left|\sum_{i=1}^N \sigma_i X_i(t-t^*)\right| = \frac{r}{\sqrt{\bE X^2}}\bE\left|\sum_{i=1}^N \sigma_i \eps_i(1+R\eta_i)\right|\\
&\leq \frac{r \sqrt{2}}{17}\left(\bE\left(\sum_{i=1}^N \sigma_i\eps_i\right)^2+R^2\bE\left(\sum_{i=1}^N \sigma_i\eps_i\eta_i\right)^2\right)^{1/2}\leq \frac{r \sqrt{2}}{17}\sqrt{N+R^2N\delta} =  r \sqrt{\frac{2N}{17}}.
\end{align*}
As a consequence, $\tilde r_2(\gamma)=(1/\gamma)\sqrt{2/(17N)}$ satisfies \eqref{comp:par}. 
We now prove Assumption~\ref{assum:fast_rates_MOM} in this particular example. 
Let $K\in[N]$ be such that $K\geq 2(575)^2/(17\times 865)$ so that, for $C_{K,r}$ is defined in \eqref{eq:CKr} with $L=1$ and $A$ defined later,
\begin{equation*}
 C_{K,r} = A^2 \max\left(\frac{2(575)^2}{17N}, \frac{865 K}{N}\right) = \frac{A^2 865K}{N}\enspace,
 \end{equation*} 
 Let $t\in\bR$ be such that $\E[(X(t-t^*))^2] = C_{K,r}$. 
 We have to show that $P\cL_t\geq A \E[(X(t-t^*))^2]$ for some well chosen $A$ and $P\cL_t = \bE[|Y-Xt|-|Y-Xt^*|]$. 
 It follows from \eqref{eq:Risk1} that $P\cL_t = \E[g(X,Xt)-g(X,Xt^*)]$ where $g:(x,a)\in\bR^2 \mapsto \int {\bf 1}_{y \geq a} (1 - F_{Y|X=x}(y))\text{d}y + (1/2)a$ and $F_{Y|X=x}$ is the cdf of $Y$ given $X=x$. Therefore, if we denote by $F$ the cdf of $\zeta$, we have
 \begin{align*}
 P\cL_t =  \bE\int_{Xt}^{Xt^*}(1-F_{Y|X=X}(y))dy =  & \frac{1-\delta}{2}\int_{t-t^*}^0 (1-F(y))dy + \frac{1-\delta}{2}\int_{t^*-t}^0 (1-F(y))dy\\ 
 &+\frac{\delta}{2}\int_{(1+R)(t-t^*)}^0(1-F(y))dy + \frac{\delta}{2}\int_{(1+R)(t^*-t)}^0(1-F(y))dy.
 \end{align*}Let us choose $K$ such that $\sqrt{C_{K,r}}\leq \sqrt{\bE X^2}\delta^\prime$ (which holds for instance when $865A^2K\leq 17 x/128$). In that case, $|t-t^*|\leq \delta^\prime$ and so $(1-F(y))=(1/2 - y)$ for all $y\in[-|t-t^*|, |t-t^*|]$. We therefore have
 \begin{equation*}
 \frac{1-\delta}{2}\int_{t-t^*}^0 (1-F(y))dy + \frac{1-\delta}{2}\int_{t^*-t}^0 (1-F(y))dy = \frac{(1-\delta)(t-t^*)^2}{2}.
 \end{equation*}Moreover, since $(1+R)|t-t^*| =(1+4\sqrt{xN})\sqrt{C_{K,r}/\bE X^2}\leq 5\sqrt{xN}\delta^\prime = 5x/(8\sqrt{2})$, we have
 \begin{equation*}
 \frac{\delta}{2}\int_{(1+R)(t-t^*)}^0(1-F(y))dy + \frac{\delta}{2}\int_{(1+R)(t^*-t)}^0(1-F(y))dy\geq  \frac{-10x\delta}{8\sqrt{2}}= \frac{-5}{4\sqrt{2}N}.
  \end{equation*}
Assume that $(1/16)^2 865K\geq 18*40/\sqrt{2}$, so $|t-t^*|\geq C_{K,r}/18\geq 40/(\sqrt{2}N)$.
Then, we have
\begin{equation*}
P\cL_t\geq (1-\delta)(t-t^*)^2/2 - 5/(4\sqrt{2}N)\geq (t-t^*)^2/16\enspace.
\end{equation*} 
For $A=1/(16\times 18)$, this yields $P\cL_t\geq A \bE[X(t-t^*)^2]$, which concludes the proof of the Berntein's assumption. 

\section{Simulation study}\label{sec:simu}

This section provides a short simulation study that illustrates our theoretical findings for the minmax MOM estimators. Let us consider the following setup: $X = (\xi_1,\cdots,\xi_d)$, where $(\xi_j)_{j=1}^d$ are independent and identically distributed, with $\xi_1  \sim \mathcal{T}(5) $, and
  \begin{align*}
    \log \bigg( \frac{\mathbf{P}(Y=1|X)}{\mathbf{P}(Y=-1|X)}  \bigg) = \inr{X,t^{*}} + \epsilon
  \end{align*} 
 where $\epsilon \sim \cL\cN(0,1)$. Let $(X_i,Y_i)_{i=1}^N$ be i.i.d with the same distribution as $(X,Y)$. We the study the minmax MOM estimator defined as:
 \begin{equation}\label{lr}
 \hat{t}^{\text{MOM}}_K  \in \arg \min_{t \in \mathbb{R}^p}  \sup_{\tilde{t} \in \mathbb{R}^p} \text{MOM}_K(\ell_t - \ell_{\tilde{t}}) \enspace. 
 \end{equation}
 Following \cite{lecue2017robust}, a gradient ascent-descent step is performed on the empirical incremental risk $(t,\tilde{t})\to P_{B_k}(\ell_t-\ell_{\tilde t})$ constructed on the block $B_k$ of data realizing the median of the empirical incremental risk.
Initial points $t_0\in\R^d$ and $\tilde{t}_0\in\R^d$ are taken at  random.  
In logistic regression, the step sizes $\eta$ and $\tilde{\eta}$ are usually chosen equal to $\|\bX\bX^\top\|_{\text{op}}/4N$, where $\bX$ is the $N\times d$ matrix with row vectors equal to $X_1^\top, \cdots, X_N^\top$ and $\|\cdot\|_{\text{op}}$ denotes the operator norm. 
In a corrupted environment, this choice might lead to disastrous performance. 
This is why $\eta$ and $\tilde{\eta}$ are computed at each iteration using only data in the median block: let $B_k$ denote the median block at the current step, then one chooses $\eta =\tilde \eta = \|\bX_{(k)}\bX_{(k)}^\top\|_{\text{op}}/4|B_k|$ where $\bX_{(k)}$ is the $|B_k| \times p$ matrix with rows given by $X_i^T$ for $i \in B_k$. 
%
In practice, $K$ is chosen by robust cross-validation choice as in \cite{lecue2017robust}.

In a first approach and according to our theoretical results, the blocks are chosen at the beginning of the algorithm. As illustrated in Figure~\ref{Fig:Ill}, this first strategy has some limitations. To understand the problem, for all $k=1,\ldots, K$, let $C_k$ denote the following set
\[
C_k = \left\{t\in\R^d: P_{B_k}\ell_t=\text{Median}\left\{P_{B_1}\ell_t,\ldots,P_{B_K}\ell_t\right\}\right\}\enspace.
\]
If the minimum of $t\to P_{B_k}\ell_t$ lies in $C_k$, the algorithm typically converges to this minimum if one iteration enters $C_k$. As a consequence, when the minmax MOM estimator \eqref{lr} lies in another cell, the algorithm does not converge to this estimator. 

To bypass this issue, the partition is changed at every ascent/descent steps of the algorithm, it is chosen uniformly at random among all equipartition of the dataset. This alternative algorithm is described in Algorithm~\ref{Alg:RandBlocks}. In practice, changing the partition seems to widely accelerate the convergence (see Figure~\ref{Fig:Ill}).

\begin{algorithm}[h]  \label{algo2}
  \SetAlgoLined
  \KwIn{The number of block $K$, initial points $t_0$ and $\tilde{t}_0$ in $\mathbb{R}^p$ and the stopping criterion $\epsilon >0$}
  \KwOut{ An estimator of $t^*$}
%
  \While{$\|t_{i}-\tilde{t}_i\|_{2} \geq \epsilon$}{
    Split the data into $K$ disjoint blocks $(B _k)_{k \in \{1, \cdots, K \}}$ of equal sizes chosen at random: $B_1 \cup \cdots \cup B_K = \{ 1, \cdots ,N \} $.\\
  Find $k \in [K] $ such that $\text{MOM}_K \ell_{t_i}= P_{B_k}  \ell_{t_i}$.\\
  Compute $\eta= \tilde{\eta} =  \|\bX_{(k)}^T \bX_{(k)}\|_{op}/4N$.\\
  Update $t_{i+1}  = t_i - \frac{1}{\eta}\nabla_t(P_{B_k}\ell_t)_{|t=t_i} $ 
  and 
  $\tilde{t}_{i+1}  = \tilde{t}_i - \frac{1}{\tilde{\eta}}\nabla_{\tilde{t}}(P_{B_k}\ell_{\tilde{t}})_{|\tilde{t}=\tilde{t}_i} $.
  }
  \caption{Descent-ascent gradient method with  blocks of data chosen at random at every steps.\label{Alg:RandBlocks}}
\end{algorithm}

 Simulation results are gathered in Figure~\ref{Fig:Ill}. In these simulations, there is no outlier, $N=1000$ and $d=100$ with $(X_i,Y_i)_{i=1}^{1000}$ i.i.d with the same distribution as $(X,Y)$. 
 Minmax MOM estimators \eqref{lr} are compared with the Logistic Regression algorithm from the  scikit-learn library of \cite{scikit-learn}.

 The upper pictures compare performance of MOM ascent/descent algorithms with fixed and changing blocks. These pictures give an example where the fixed block algorithm is stuck into local minima and another one where it does not converge. In both cases, the changing blocks version converges to $t^*$. 

  Running times of logistic regression (LR) and its MOM version (MOM LR) are compared in the lower picture of Figure~\ref{Fig:Ill} in a dataset free from outliers. 
  LR and MOM LR are coded with the same algorithm in this example, meaning that MOM gradient descent-ascent and simple gradient descent are performed with the same descent algorithm. 
As illustrated in Figure~\ref{Fig:Ill}, running each step of the gradient descent on one block only and not on the whole dataset accelerates the running time. The larger the dataset, the bigger the benefit is expected. 

\begin{figure}[h]
\begin{minipage}{\linewidth}
  \centering
  \begin{minipage}{0.45\linewidth}
      \includegraphics[width=\linewidth]{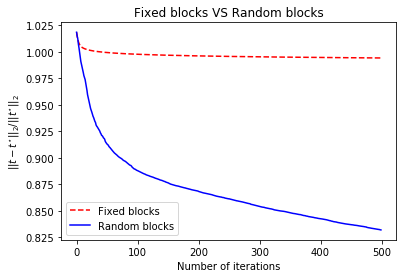}\label{fig:LogRegNR}
  \end{minipage}
  \hspace{0.05\linewidth}
  \begin{minipage}{0.45\linewidth}
      \includegraphics[width=\linewidth]{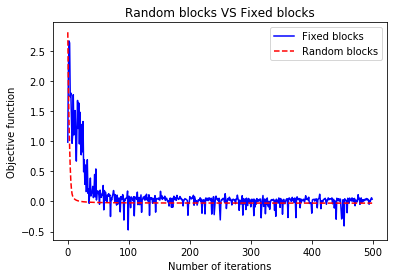}
  \end{minipage}
  \begin{minipage}{0.45\linewidth}
    \includegraphics[width=\linewidth]{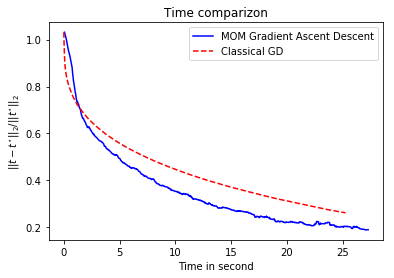}
  \end{minipage}
\end{minipage}
\caption{Top left and right: Comparizon of the algorithm with fixed and changing blocks. Bottom: Comparizon of running time between classical gradient descent and algorithm~\ref{Alg:RandBlocks}. In all simulation $N=1000$, $p=100$ and there is no outliers.\label{Fig:Ill}}
\end{figure}

The resistance to outliers of logistic regression and its minmax MOM alternative are depicted in Figure~\ref{fig:MOMVsERMLogReg} in the introduction. We added an increasing number of outliers to the dataset. Outliers $\{(X_i,Y_i),i\in\cO\}$ in this simulation are such that $X_i \sim \cL\cN(0,5)$ and $Y_i =  -sign(\inr{X_i,t}+ \epsilon_i)$, with $\epsilon_i\sim\epsilon$ as above. Figure~\ref{fig:MOMVsERMLogReg} shows that logistic classification is mislead by a single outlier while MOM version maintains reasonable performance with up to 50 outliers (i.e $5$\% of the database is corrupted).

A byproduct of Algorithm~\ref{Alg:RandBlocks} is an outlier detection algorithm. Each data receives a score equal to the number of times it is selected in a median block in the random choice of block version of the algorithm. 
The first iterations may be misleading: before convergence, the empirical loss at the current point may not reveal the centrality of the data because the current point may be far from  $t^*$.
Simulations are run with $N=100$, $d = 10$ and $5000$ iterations and therefore only the score obtained by each data in the last $4000$ iterations are displayed. $3$ outliers $(X_i,Y_i)_{i \in \{1,2,3\}}$ with $X_i = (10)_{j=1}^d$ and $Y_i = -sign(\inr{X_i,t})$ have been introduced at number $42$, $62$ and $66$. Figure~\ref{fig:outliers_detect} shows that these are not selected once.  

 \begin{figure}[h] 
	\includegraphics[width=\linewidth,scale=0.5]{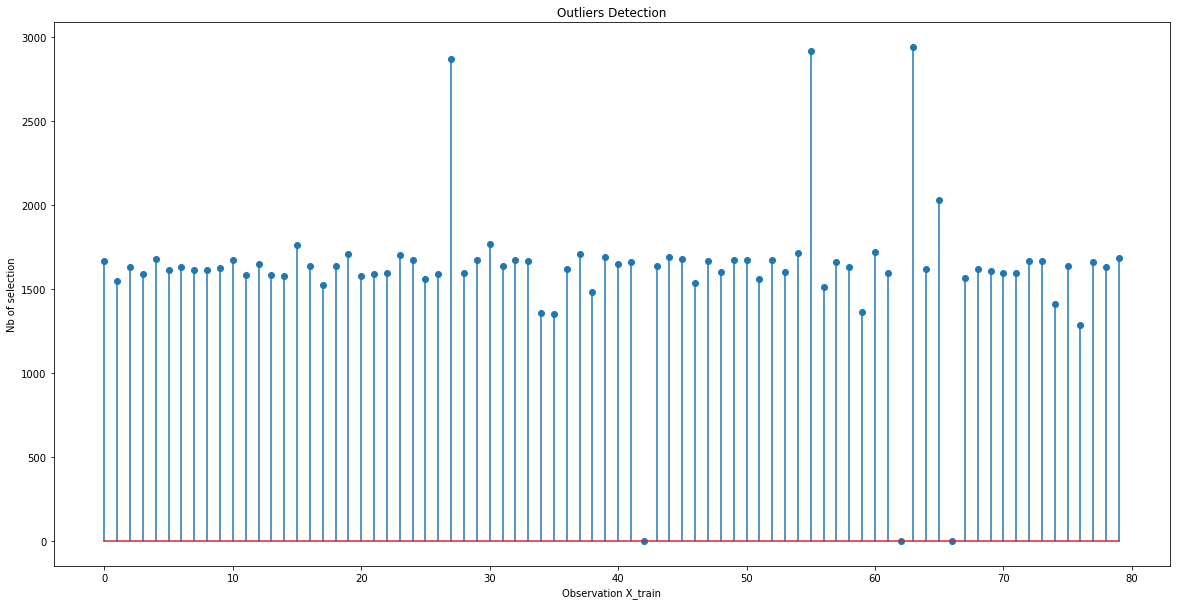}
	\caption{Outliers Detection Procedure for $N=100$, $p=10$ and outliers are $i=42, 62, 66$ \label{fig:outliers_detect}}
\end{figure}


\section{Conclusion} 
The paper introduces a new homogenity argument for learning problems with convex and Lipschitz losses. 
This argument allows to obtain estimation rates and oracle inequalities for ERM and minmax MOM estimators improving existing results.
The ERM requires sub-Gaussian hypotheses on the class $F$ with respect to the distribution of the design and a local Bernstein condition (see Theorem~\ref{theo:erm}), both assumptions can be removed for minmax MOM estimators (see Theorem~\ref{theo:bernstein_kappa}). 
The local Bernstein conditions provided in this article can be verified in several learning problems. 
In particular, it allows to derive optimal risk bounds in examples where analyses based on the small ball hypothesis fail. 
Minmax MOM estimators applied to convex and Lipschitz losses are efficient under weak assumptions on the outputs $Y$, under minimal $L_2$ assumptions on the class $F$ with respect to the distribution of the design and the results are robust to the presence of few outliers in the dataset. 
A modification of these estimators can be implemented efficiently and confirm all these conclusions.

\appendix
\section{Proof of Theorems~\ref{theo:erm},~\ref{theo:main}, \ref{theo:MOM_lepski} and \ref{theo:main_without_bernstein_cond}} \label{app:proof}

\subsection{Proof of Theorem~\ref{theo:erm}} \label{proof_erm} 

The proof is splitted in two parts. First, we identify an event where the statistical behavior of the regularized estimator $\hat{f}^{ERM}$ can be controlled. Then, we prove that this event holds with probability at least \eqref{eq:proba}. Introduce $\theta=1/(2A)$ and define the following event:
	\begin{equation*}
	\Omega := \left\{ 
	\forall f\in F \cap  (f^{*} + r_2(\theta ) B_{L_2}), \quad \big|(P-P_N)\cL_f\big|\leq \theta r_2^2(\theta )  \right\}
	\end{equation*} 
	where $\theta$ is a parameter appearing in the definition of $r_2$ in Definition~\ref{def:function_r}.
	
	\begin{Proposition}\label{prop:algebra} On the event $\Omega$, one has 
		\begin{align*}
		\|\hat{f}^{ERM} - f^*\|_{L_2} & \leq r_2(\theta ) \mbox{ and }
		P\cL_{\hat{f}^{ERM}}  \leq  \theta r_2^2(\theta ).
		\end{align*}
	\end{Proposition} 

	\begin{proof}
		By construction, $\hat{f}^{ERM}$ satisfies $P_N\cL_{\hat{f}^{ERM}}  \leq 0 $. Therefore, it is sufficient to show that, on $\Omega$, if $\|f-f^{*}\|_{L_2} > r_2(\theta )$, then $P_N \cL_f >0$. 
%
		Let $f\in F$ be such that $\|f-f^{*}\|_{L_2 } > r_2(\theta )$. By convexity of $F$, there exists $f_0 \in  F \cap (f^{*} + r_2(\theta)S_{L_2})$ and $\alpha > 1$ such that
		\begin{equation} \label{homo_argument}
		f = f^{*} + \alpha (f_0 - f^{*}) \enspace. 
		\end{equation}
		For all $i \in \{1,\cdots,N \}$, let $\psi_i: \mathbb R \rightarrow \mathbb R $ be defined for all $u\in \R$ by 
		\begin{equation}\label{eq:fct_psi}
		\psi_i(u) = \overline{\ell} (u + f^{*}(X_i), Y_i) - \overline{\ell} (f^{*}(X_i), Y_i).
		\end{equation}
		The functions $\psi_i$ are such that $\psi_i(0) = 0$, they are convex because $\overline{\ell}$ is, in particular $\alpha\psi_i(u) \leq \psi_i(\alpha u)$ for all $u\in\mathbb R$ and $\alpha \geq 1$ and $\psi_i(f(X_i) - f^{*}(X_i) )=  \overline{\ell} (f(X_i), Y_i) - \overline{\ell} (f^{*}(X_i), Y_i) $ so that the following holds:
		\begin{align} \label{conv_arg}
		\nonumber P_N \cL_f & = \frac{1}{N} \sum_{i=1}^{N}  \psi_i \big( f(X_i)- f^{*}(X_i) \big) = \frac{1}{N} \sum_{i=1}^{N}  \psi_i(\alpha( f_0(X_i)- f^{*}(X_i) ))\\
		&\geq \frac{\alpha}{N} \sum_{i=1}^{N}   \psi_i(( f_0(X_i)- f^{*}(X_i))) = \alpha P_N \cL_{f_0}.
		\end{align}

		Until the end of the proof, the event $\Omega$ is assumed to hold. Since $f_0 \in F \cap (f^{*}+  r_2(\theta ) S_{L_2})$,  $P_N \cL_{f_0} \geq P\cL_{f_0} - \theta r_2^2(\theta )$.
		Moreover, by Assumption~\ref{assum:fast_rates}, $P\cL_{f_0} \geq A^{-1} \|f_0-f^*\|_{L_2 }^2 = A^{-1}r_2^2(\theta) $, thus
		\begin{equation} \label{f_0}
		P_N \cL_{f_0} \geq (A^{-1} - \theta ) r_2^2(\theta). 
		\end{equation}
		From Eq.~\eqref{conv_arg} and \eqref{f_0}, $P_N \cL_f > 0$ since $A^{-1}>\theta$. Therefore, $\|\hat{f}^{ERM}-f^{*}\|_{L_2 } \leq r_2^2(\theta )$. This proves the $L_2$-bound. 

		Now, as $\|\hat{f}^{ERM}-f^{*}\|_{L_2 } \leq r_2^2(\theta )$, $|(P-P_N)\cL_{\hat{f}^{ERM}}|\leq \theta r_2^2(\theta )$. Since $P_N\cL_{\hat{f}^{ERM}}\leq 0$, 
		\begin{equation*}
		P\cL_{\hat{f}^{ERM}} = P_N\cL_{\hat{f}^{ERM}} + (P-P_N)\cL_{\hat{f}^{ERM}}\leq \theta r_2^2(\theta )\enspace.
		\end{equation*}
		This show the excess risk bound.
	\end{proof}

	Proposition~\ref{prop:algebra} shows that $\hat{f}^{ERM}$ has the risk bounds given in Theorem~\ref{theo:erm} on the event $\Omega$. 
	To show that $\Omega$ holds with probability \eqref{eq:proba}, recall the following results from \cite{pierre2017estimation}.
	\begin{Lemma} \cite{pierre2017estimation} \label{lem:subgauss}[Lemma 8.1]
		Grant Assumptions~\ref{assum:lip_conv} and~\ref{ass:sub-gauss}. Let $F^\prime\subset F$ with finite $L_2$-diameter $d_{L_2}(F^\prime)$. For every $u>0$, with probability at least $1-2\exp(-u^2)$,
		\begin{equation*}
		\sup_{f,g\in F^\prime}\left|(P-P_N)(\cL_f-\cL_g)\right|\leq \frac{16L}{\sqrt{N}} \left(w(F^\prime) +   u d_{L_2}(F^\prime)\right)\enspace.
		\end{equation*} 
	\end{Lemma}
	It follows from Lemma~\ref{lem:subgauss} that for any $u>0$, with probability larger that $1-2\exp(-u^2)$,
	\begin{align*}
	&\sup_{f \in F \cap (f^{*} + r_2(\theta) B_{L_2})} \big | (P-P_N)\cL_f  \big|  \leq \sup_{f,g \in  F \cap (f^{*} + r_2(\theta) B_{L_2})} \big | (P-P_N)(\cL_f-\cL_g)  \big| \\
	& \leq \frac{16L}{\sqrt{N}} \big(   w((F-f^*)\cap r_2(\theta)B_{L_2}) + ud_{L_2} ((F-f^*)\cap r_2(\theta)B_{L_2})   \big)
	\end{align*}
	where $d_{L_2} ((F-f^*)\cap r_2(\theta)B_{L_2}) \leq r_2(\theta)$. By definition of the complexity parameter (see Eq.~\eqref{def:function_r}), for $u = \theta \sqrt{N} r_2(\theta )/(64L) $, with probability at least 
	\begin{equation}
	1-2\exp\big(-\theta^2N r_2^2(\theta ) /(16^3L^2 ) \big)\enspace,
	\end{equation} 
	for every $f$ in $F\cap(f^*+ r_2(\theta)B_{L_2} )$,
	\begin{equation}
	\big | (P-P_N) \cL_f \big|  \leq \theta  r_2^2(\theta ).
	\end{equation}
	Together with Proposition~\ref{prop:algebra}, this concludes the proof of Theorem~\ref{theo:erm}.

\subsection{Proof of Theorem~\ref{theo:main} } \label{proof_MOM}
The proof is splitted in two parts. First, we identify an event $\Omega_K$ where the statistical properties of $\hat f$ from Theorem~\ref{theo:main}  can be established. Next, we prove that this event holds with probability \eqref{eq:proba_MOM}. Let $\alpha, \theta$ and $\gamma$ be positive numbers to be chosen later. Define 
\begin{equation*}
C_{K,r} = \max \bigg(\frac{4L^2K}{\theta^2 \alpha N},\tilde{r}_2^2(\gamma) \bigg) 
\end{equation*} 
where the exact form of $\alpha, \theta$ and $\gamma$ are given in Equation~\eqref{choice_constants}. Set the event $\Omega_K$ to be such that
\begin{equation}\label{eq:event_omegaK} 
\Omega_K = \bigg \{ \forall f \in F \cap \left(f^*+ \sqrt{C_{K,r}}B_{L_2}\right), \exists J\subset\{1,\ldots,K\}: |J|>K/2 \mbox{ and } \forall k\in J, \left| (P_{B_k} - P)\cL_f \right| \leq \theta C_{K,r} \bigg \}\enspace.
\end{equation}

\subsubsection{Deterministic argument}
The goal of this section is to show that, on the event $\Omega_K$, $\|\hat f - f^{*}\|_{L_2}^2  \leq C_{K,r}$ and $P\cL_{\hat f}\leq 2 \theta C_{K,r}$. 
\begin{Lemma}\label{lem:lem1_MOM}
	If there exists $\eta>0$ such that  
	\begin{align}\label{obj_proof}
	\sup_{f \in F \backslash \left(f^{*} +  \sqrt{C_{K,r}}B_{L_2}\right) }&  \hspace{0.2cm} \MOM{K}{\ell_{f^{*}}-\ell_f} < - \eta \quad \mbox{  and }  \sup_{f \in F\cap\left(f^{*}+  \sqrt{C_{K,r}} B_{L_2}\right)} \MOM{K}{\ell_{f^{*}}-\ell_f} \leq  \eta \enspace,
	\end{align}
	then $\|\hat f - f^{*} \|_{L_2 }^2 \leq  C_{K,r}$.
\end{Lemma}
\begin{proof}
	Assume that \eqref{obj_proof} holds, then  
	\begin{equation}
	\inf_{f\in F \backslash \left(f^{*} +  \sqrt{C_{K,r}}B_{L_2}\right)} \text{MOM}_K[\ell_f-\ell_{f^*}]> \eta\enspace.\label{eq:Task2} 
	\end{equation}
	Moreover, if $T_K(f)=\sup_{g\in F}\text{MOM}_K[\ell_f-\ell_g]$ for all $f\in F$, then
	\begin{equation}
	T_K(f^{*})=\sup_{f\in  F \cap \left(f^{*}+ \sqrt{C_{K,r}}B_{L_2}\right) }\text{MOM}_K[\ell_{f^*}-\ell_f]\vee \sup_{f\in F \backslash \left(f^{*} +  \sqrt{C_{K,r}}B_{L_2}\right) }\text{MOM}_K[\ell_{f^*}-\ell_f]\leqslant \eta\enspace.\label{eq:Task1}
	\end{equation}
By definition of $\hat{f}$ and \eqref{eq:Task1}, $T_K(\hat{f})\leqslant T_K(f^*)\leqslant \eta$.
	Moreover, by \eqref{eq:Task2}, any $f\in F \backslash \left(f^{*} +  \sqrt{C_{K,r}}B_{L_2}\right)$ satisfies $T_K(f)\geqslant \text{MOM}_K[\ell_f-\ell_{f^*}]> \eta$.
	Therefore $\hat{f} \in  F\cap (f^{*} + \sqrt{C_{K,r}} B_{L_2})$.	
\end{proof} 


\begin{Lemma}\label{lem:lem2_MOM}
Grant Assumption~\ref{assum:fast_rates_MOM} and assume that $\theta-A^{-1}<-\theta$. On the event $\Omega_K$, \eqref{obj_proof} holds with $\eta = \theta C_{K,r}$.
\end{Lemma}

\begin{proof}
Let $f\in F$ be such that $\|f-f^{*}\|_{L_2 } > C_{K,r}$. By convexity of $F$, there exists $f_0 \in    F \cap \left(f^{*} + \sqrt{C_{K,r}} S_{L_2}\right) $ and $\alpha > 1$ such that $f = f^{*} + \alpha (f_0 - f^{*})$.
%
For all $i \in \{1,\ldots,N \}$, let $\psi_i: \mathbb R \rightarrow \mathbb R $ be defined for all $u\in \R$ by 
\begin{equation}\label{eq:fct_psi_MOM}
\psi_i(u) = \overline{\ell} (u + f^{*}(X_i), Y_i) - \overline{\ell} (f^{*}(X_i), Y_i).
\end{equation}
The functions $\psi_i$ are convex because $\overline{\ell}$ is and such that $\psi_i(0) = 0$, so $\alpha\psi_i(u) \leq \psi_i(\alpha u)$ for all $u\in\mathbb R$ and $\alpha \geq 1$. As $\psi_i(f(X_i) - f^{*}(X_i) )=  \overline{\ell} (f(X_i), Y_i) - \overline{\ell} (f^{*}(X_i), Y_i)$, for any block $B_k$, 
\begin{align} \label{conv_arg_MOM}
\nonumber P_{B_k} \cL_f & = \frac{1}{|B_k|} \sum_{i \in B_k}  \psi_i \big( f(X_i)- f^{*}(X_i) \big)= \frac{1}{|B_k|} \sum_{i \in B_k}  \psi_i(\alpha( f_0(X_i)- f^{*}(X_i) ))\\
&\geq \frac{\alpha}{|B_k|} \sum_{i \in B_k}   \psi_i(( f_0(X_i)- f^{*}(X_i))) = \alpha P_{B_k} \cL_{f_0}.
\end{align}

As $f_0 \in F \cap (f^* +  \sqrt{C_{K,r}} S_{L_2})$, on $\Omega_K$, there are strictly more than $K/2$ blocks $B_k$ where $P_{B_k} \cL_{f_0} \geq P\cL_{f_0} - \theta C_{K,r}$. Moreover, from Assumption~\ref{assum:fast_rates_MOM}, $P\cL_{f_0} \geq A^{-1} \|f_0-f^*\|_{L_2 }^2 = A^{-1}C_{K,r} $. Therefore, on strictly more than $K/2$ blocks $B_k$, 
\begin{equation} \label{f_0_MOM}
P_{B_k} \cL_{f_0} \geq (A^{-1} - \theta ) C_{K,r}.
\end{equation}
From Eq.~\eqref{conv_arg_MOM} and ~\eqref{f_0_MOM}, there are strictly more than $K/2$ blocks $B_k$ where $P_{B_k} \cL_f \geq (A^{-1}- \theta) C_{K,r} $. Therefore, on $\Omega_K$, as $(\theta - A^{-1}) < - \theta$,
\begin{equation*}
	\sup_{f \in  F \backslash \left(f^{*} +  \sqrt{C_{K,r}}B_{L_2}\right)}  \hspace{0.2cm} \MOM{K}{\ell_{f^{*}}-\ell_f} < (\theta - A^{-1}) C_{K,r}<-\theta C_{K,r}\enspace.
\end{equation*}

In addition, on the event $\Omega_K$, for all $f \in   F \cap (f^{*} + \sqrt{C_{K,r}}B_{L_2})$, there are strictly more than $K/2$ blocks $B_k$ where $|(P_{B_k}-P) \cL_f | \leq  \theta C_{K,r} $. Therefore   
\begin{equation*}
\MOM{K}{\ell_{f^{*}}-\ell_f} \leq \theta C_{K,r} - P\cL_f \leq  \theta C_{K,r}.
\end{equation*}
\end{proof}


\begin{Lemma}\label{lem:lem3_MOM}
Grant Assumption~\ref{assum:fast_rates_MOM} and assume that $\theta - A^{-1}<-\theta$. On the event $\Omega_K$, $P\cL_{\hat{f}} \leq 2\theta C_{K,r}$.
\end{Lemma}
\begin{proof}Assume that $\Omega_K$ holds. From Lemmas~\ref{lem:lem1_MOM} and~\ref{lem:lem2_MOM}, $\|\hat{f}-f^{*}\|_{L_2 } \leq \sqrt{C_{K,r}}$.  Therefore, on strictly more than $K/2$ blocks $B_k$, $P \cL_{\hat{f}} \leq P_{B_k} \cL_{\hat{f}} + \theta C_{K,r}$. In addition, by definition of $\hat f$ and \eqref{eq:Task1} (for $\eta = \theta C_{K,r}$),
\begin{equation*}
MOM_K(\ell_{\hat f} - \ell_{f^{*}}) \leq \sup_{f \in F} MOM_K(\ell_{f^{*}} - \ell_{f}) \leq \theta C_{K,r}.
\end{equation*}
As a consequence, there exist at least $K/2$ blocks $B_k$ where $P_{B_k} \cL_{\hat{f}} \leq \theta C_{K,r}$. Therefore, there exists at least one block $B_k$ where both $P \cL_{\hat{f}} \leq P_{B_k} \cL_{\hat{f}} + \theta C_{K,r}$ and  $P_{B_k} \cL_{\hat{f}} \leq \theta C_{K,r}$. Hence $P\cL_{\hat{f}} \leq 2\theta C_{K,r}$. 
\end{proof}

\subsubsection{Stochastic argument}
This section shows that $\Omega_K$ holds with probability at least \eqref{eq:proba_MOM}. 
\begin{Proposition}\label{prop:sto_MOM}
	Grant Assumptions~\ref{assum:lip_conv},~\ref{assum:convex},~\ref{assum:moments} and~\ref{assum:fast_rates_MOM} and assume that $(1-\beta)K\geq |\cO|$. Let $x>0$ and assume that $\beta(1-\alpha-x-8\gamma L/\theta)>1/2$. Then $\Omega_K$ holds with probability larger than $1-\exp(-x^2 \beta K/2)$.
\end{Proposition}

\begin{proof}
	Let $\cF = F \cap \left(f^{*} + \sqrt{C_{K,r}}B_{L_2}\right)$ and set $\phi:t\in\R\to I \{ t\geq 2 \} + (t-1) I \{1 \leq t \leq 2 \}$ so, for all $t \in \mathbb{R}$, $I \{ t\geq 2 \} \leq \phi(t) \leq I \{ t\geq 1 \}$. Let $W_k = ((X_i,Y_i))_{i \in B_k}$, $G_f(W_k) = (P_{B_k} - P)\cL_f$.
Let
	\begin{align*}
	z(f) & = \sum_{k =1}^K I \{|G_f(W_k)|\leq \theta C_{K,r} \}.
	\end{align*}
	 
	 Let $\mathcal{K}$ denote the set of indices of blocks which have not been corrupted by outliers, $\mathcal{K} = \{k \in \{1,\cdots,K \} : B_k \subset \mathcal{I}\}$ and let $f \in \cF$. Basic algebraic manipulations show that
	\begin{equation*}
	z(f) \geq |\mathcal{K}| - \sup_{f \in \cF } \sum_{k \in \mathcal{K}} \bigg( \phi(2\theta^{-1}C_{K,r}^{-1} | G_f(W_k)|) - \mathbb{E} \phi(2\theta^{-1}C_{K,r}^{-1} | G_f(W_k)|) \bigg) -  \sum_{k \in \mathcal{K} } \mathbb{E}\phi(2\theta^{-1}C_{K,r}^{-1} | G_f(W_k)|) \enspace.
	\end{equation*} 
By Assumptions~\ref{assum:lip_conv} and \ref{assum:moments}, using that $C_{K,r}^2\geq \norm{f-f^*}^2_{L_2 }[(4L^2K)/(\theta^2\alpha N)]$,
	\begin{align*}
	\mathbb{E}\phi(2\theta^{-1}C_{K,r}^{-1} | G_f(W_k)|) & \leq \mathbb{P} \bigg( |G_f(W_k)| \geq \frac{\theta C_{K,r}}{2} \bigg) \leq \frac{4}{\theta^2C_{K,r}^2} \mathbb{E}G_f(W_k)^2  = \frac{4}{\theta^2 C_{K,r}^2} \mathbb{V}ar (P_{B_k}\cL_f) \\
	&  \leq  \frac{4K^2}{\theta^2C_{K,r}^2N^2} \sum_{i \in B_k} \mathbb{E} [\cL_f^2(X_i,Y_i)] \leq \frac{4L^2K}{\theta^2C_{K,r}^2N}\|f-f^{*}\|^2_{L_2 } \leq \alpha\enspace.
	\end{align*}
	Therefore, 
	\begin{align}\label{res::1}
	z(f) \geq |\mathcal{K}|(1-\alpha) -\sup_{f \in \cF} \sum_{k \in \mathcal{K}} \bigg( \phi(2\theta^{-1}C_{K,r}^{-1} | G_f(W_k)|) - \mathbb{E} \phi(2\theta^{-1}C_{K,r}^{-1} | G_f(W_k)|) \bigg) \enspace.
	\end{align}
	Using Mc Diarmid's inequality \cite[Theorem~6.2]{MR3185193}, for all $x>0$, with probability larger than $1-\exp(-x^2 |\cK| /2)$, 
	\begin{align*}
	\sup_{f \in \cF } & \sum_{k \in \mathcal{K}} \bigg( \phi(2\theta^{-1}C_{K,r}^{-1} | G_f(W_k)|)  - \mathbb{E} \phi(2\theta^{-1}C_{K,r}^{-1} | G_f(W_k)|) \bigg)   \\
	& \leq  x|\mathcal{K}| + \mathbb{E} \sup_{f \in \cF }   \sum_{k \in \mathcal{K}} \bigg( \phi(2\theta^{-1}C_{K,r}^{-1} | G_f(W_k)|) - \mathbb{E}  \phi(2\theta^{-1}\theta^{-1}C_{K,r}^{-1} | G_f(W_k)|) \bigg)\enspace.
	\end{align*}
	Let $\eps_1, \ldots, \eps_K$ denote independent Rademacher variables independent of the $(X_i, Y_i), i\in\cI$. By Gin{\'e}-Zinn symmetrization argument,
\begin{multline*}
	\sup_{f \in \cF }  \sum_{k \in \mathcal{K}} \bigg( \phi(2\theta^{-1}C_{K,r}^{-1} | G_f(W_k)|)  - \mathbb{E} \phi(2\theta^{-1}C_{K,r}^{-1} | G_f(W_k)|) \bigg)   \\
	\leq x|\mathcal{K}| + 2 \mathbb{E} \sup_{f \in \cF }   \sum_{k \in \mathcal{K}} \eps_k \phi(2\theta^{-1}C_{K,r}^{-1} | G_f(W_k)|) 
	\end{multline*}
As $\phi$ is 1-Lipschitz with $\phi(0)=0$, using the contraction lemma \cite[Chapter~4]{ledoux2013probability}, 
	\begin{align*}
\mathbb{E} \sup_{f \in \cF }   \sum_{k \in \mathcal{K}} \eps_k \phi(2\theta^{-1}C_{K,r}^{-1} | G_f(W_k)|)\leq 	2\mathbb{E}  \sup_{f \in \cF }    \sum_{k \in \mathcal{K}}  \eps_k \frac{ G_f(W_k)}{\theta C_{K,r}}   =  2\mathbb{E} \sup_{f \in \cF } \sum_{k \in \mathcal{K}}  \eps_k \frac{(P_{B_k}- P)\cL_f}{\theta C_{K,r}}. 
	\end{align*}
	Let $(\sigma_i: i \in \cup_{k\in\cK}B_k)$ be a family of independent Rademacher variables independent of $(\eps_k)_{k \in \mathcal{K}}$ and $(X_i, Y_i)_{i \in \cI}$. It follows from the Gin{\'e}-Zinn symmetrization argument that
\begin{equation*}
\mathbb{E} \sup_{f \in \cF } \sum_{k \in \mathcal{K}}  \eps_k \frac{(P_{B_k}- P)\cL_f}{ C_{K,r}}\leq 2 \mathbb{E} \sup_{f \in \cF } \frac{K}{N}\sum_{i \in \cup_{k\in\cK}B_k }  \sigma_i \frac{\cL_f(X_i,Y_i)}{ C_{K,r}}.
\end{equation*}
By the Lipschitz property of the loss, the contraction principle applies and
	\begin{align*}
	  \mathbb{E} \sup_{f \in \cF } \sum_{i \in \cup_{k\in\cK}B_k }  \sigma_i \frac{\cL_f(X_i,Y_i)}{ C_{K,r}} \leq  L\mathbb{E} \sup_{f \in \cF} \sum_{i \in \cup_{k \in \mathcal{K}} B_k}  \sigma_i \frac{(f-f^{*})(X_i)}{C_{K,r}}\enspace.
	\end{align*}
To bound from above the right-hand side in the last inequality, consider two cases 1) $C_{K,r}= \tilde{r}_2^2(\gamma)$ or 2) $C_{K,r} = 4L^2K/(\alpha\theta^2 N)$. In the first case, by definition of the complexity parameter $\tilde{r}_2(\gamma)$ in \eqref{comp:par},
	\begin{align*}
	 \mathbb{E} \sup_{f \in \cF }   \sum_{i \in \cup_{k \in \mathcal{K}} B_k}  \sigma_i \frac{(f-f^{*})(X_i)}{C_{K,r}} 
	  =  \mathbb{E} \sup_{f \in F: \|f-f^{*}\|_{L_2 } \leq \tilde r_2(\gamma) } \frac{1}{\tilde{r}_2^2(\gamma)} \bigg| \sum_{i \in \cup_{k \in \mathcal{K}} B_k}  \sigma_i  (f-f^{*})(X_i)\bigg| \leq \frac{\gamma |\cK| N}{K}.
	\end{align*}
	In the second case, 
	\begin{align*}
	&\mathbb{E} \sup_{f \in \cF } \sum_{i \in \cup_{k \in \mathcal{K}} B_k}   \frac{\sigma_i(f-f^{*})(X_i)}{C_{K,r}} \\
	& \leq  	\mathbb{E}  \bigg[  \sup_{\substack{f \in F:\\ \|f-f^{*}\|_{L_2 } \leq \tilde{r}_2(\gamma)}}  \bigg| \sum_{i \in \cup_{k \in \mathcal{K}} B_k}  \frac{\sigma_i (f-f^{*})(X_i)}{\tilde{r}_2^2(\gamma)} \bigg|  \vee  \sup_{\substack{f \in F:\\  \tilde{r}_2(\gamma)  \leq \|f-f^{*}\|_{L_2 } \leq \sqrt{\frac{4L^2K}{\alpha \theta^2 N}}} }  \bigg| \sum_{i \in \cup_{k \in \mathcal{K}} B_k}  \sigma_i \frac{(f-f^{*})(X_i)}{\frac{4L^2K}{\alpha \theta^2 N}} \bigg|   \bigg] \enspace.
	\end{align*}
Let $f\in F$ be such that $\tilde{r}_2(\gamma) \leq \norm{f-f^*}_{L_2 }\leq \sqrt{[4L^2K]/[\alpha\theta^2 N]}$; by convexity of $F$, there exists $f_0\in F$ such that $\norm{f_0-f^*}_{L_2 } = \tilde{r}_2(\gamma)$ and $f = f^*+\alpha(f_0-f^*)$ with $\alpha = \norm{f-f^*}_{L_2 }/\tilde{r}_2(\gamma)\geq 1$. Therefore,
	\begin{align*}
	\bigg| \sum_{i \in \cup_{k \in \mathcal{K}} B_k}  \sigma_i \frac{(f-f^{*})(X_i)}{\frac{4L^2K}{\alpha \theta^2 N}} \bigg|  \leq 	\frac{1}{\tilde{r}_2(\gamma ) }  \bigg| \sum_{i \in \cup_{k \in \mathcal{K}} B_k}  \sigma_i \frac{(f-f^{*})(X_i)}{\|f-f^{*}\|_{L_2 }} \bigg|  = \frac{1}{\tilde{r}_2^2(\gamma ) }  \bigg| \sum_{i \in \cup_{k \in \mathcal{K}} B_k}  \sigma_i (f_0-f^{*})(X_i) \bigg| 
	\end{align*} and so 
\begin{equation*}
\sup_{\substack{f \in F:\\  \tilde{r}_2(\gamma)  \leq \|f-f^{*}\|_{L_2 } \leq \sqrt{\frac{4L^2K}{\alpha \theta^2 N}} }}   \bigg| \sum_{i \in \cup_{k \in \mathcal{K}} B_k}  \sigma_i \frac{(f-f^{*})(X_i)}{\frac{4L^2K}{\alpha \theta^2 N}} \bigg| \leq \frac{1}{\tilde{r}_2^2(\gamma ) } \sup_{\substack{f \in F:\\  \|f-f^{*}\|_{L_2 } = \tilde{r}_2(\gamma) } } \bigg| \sum_{i \in \cup_{k \in \mathcal{K}} B_k}  \sigma_i (f-f^{*})(X_i) \bigg|.
\end{equation*}
By definition of $\tilde{r}_2(\gamma)$, it follows that
	\begin{align*}
		\mathbb{E} \sup_{f \in \cF } \bigg| & \sum_{i \in \cup_{k \in \mathcal{K}} B_k}  \sigma_i \frac{(f-f^{*})(X_i)}{ C_{K,r}} \bigg|  \leq \frac{\gamma |\cK| N}{K}.
	\end{align*}
	Therefore, as $|\mathcal{K}| \geq K-|\mathcal{O}| \geq \beta K$, with probability larger than $1-\exp(-x^2 \beta K/2)$, for all $f\in F$ such that $\norm{f-f^*}_{L_2 }\leq \sqrt{C_{K,r}}$,
	\begin{align} \label{choice_constants}
	z(f) \geq  |\mathcal{K}|\left(1-\alpha - x - \frac{8 \gamma L}{\theta}\right) > \frac{K}{2}.
	\end{align}
\end{proof}

\subsubsection{End of the proof of Theorem~\ref{theo:main}} 
\label{sub:end_of_the_proof_of_theorem_theo:main}
Theorem~\ref{theo:main} follows from Lemmas~\ref{lem:lem1_MOM}, \ref{lem:lem2_MOM}, \ref{lem:lem3_MOM} and Proposition~\ref{prop:sto_MOM} for the choice of constant
\begin{gather*}
 \theta = 1/(3A)  \quad \alpha = 1/24, \quad x = 1/24 ,\quad  \beta = 4/7 \mbox{ and } \gamma = 1/(575 AL).
\end{gather*}

\subsection{Proof of Theorem~\ref{theo:MOM_lepski}} 
\label{sub:proof_of_theorem_theo:mom_lepski}
Let $K \in \big[ 7|\cO|/3 ,  N \big]$ and consider the event $\Omega_K$ defined in \eqref{eq:event_omegaK}. 
It follows from the proof of Lemmas~\ref{lem:lem1_MOM} and~\ref{lem:lem2_MOM} that $T_K(f^*)\leq \theta C_{K,r}$ on $\Omega_K$. 
Setting $\theta=1/(3A)$, on $\cap_{J=K}^N \Omega_J$, $f^*\in \hat R_J$ for all $J=K,\ldots, N$, so $\cap_{J=K}^N \hat R_J\neq \emptyset$. 
By definition of $\hat K$, it follows that $\hat K\leq K$ and by definition of $\tilde f$, $\tilde f \in \hat R_K$ which means that $T_K(\tilde f)\leq \theta C_{K,r}$. 
It is proved in Lemmas~\ref{lem:lem1_MOM} and~\ref{lem:lem2_MOM} that on $\Omega_K$, if $f\in F$ satisfies $\norm{f-f^*}_{L_2 }\geq \sqrt{C_{K,r}}$ then $T_K(f)> \theta C_{K,r}$. 
Therefore, $\norm{\tilde f-f^*}_{L_2 }\leq \sqrt{C_{K,r}}$. 
On $\Omega_K$, since $\norm{\tilde f-f^*}_{L_2 }\leq \sqrt{C_{K,r}}$, $P\cL_{\tilde f}\leq 2 \theta C_{K,r}$. Hence, on $\cap_{J=K}^N \Omega_J$,  the conclusions of Theorem~\ref{theo:MOM_lepski} hold. 
Finally, by Proposition~\ref{prop:sto_MOM}, 
\begin{equation*}
\bP\left[\cap_{J=K}^N \Omega_J\right]\geq 1-\sum_{J=K}^N \exp(-K/2016)\geq 1-4 \exp(-K/2016).
\end{equation*}


\subsection{Proof of Theorem~\ref{theo:main_without_bernstein_cond}} 
\label{ssub:proof_of_theorem_theo:main_without_bernstein_cond}
The proof of Theorem~\ref{theo:main_without_bernstein_cond} follows the same path as the one of Theorem~\ref{theo:main}. We only sketch the different arguments needed because of the localization by the excess loss and the lack of Bernstein condition. 

Define the event $\Omega_K^\prime$ in the same way as $\Omega_K$ in \eqref{eq:event_omegaK} where $C_{K,r}$ is replaced by $\bar{r}_2^2(\gamma)$
and the $L_2$ localization is replaced by the ``excess loss localization'':
\begin{equation}\label{eq:event_omegaK_prime} 
\Omega^\prime_K = \bigg \{ \forall f \in (\cL_F)_{\bar{r}_2^2(\gamma)}, \exists J\subset\{1,\ldots,K\}: |J|>K/2 \mbox{ and } \forall k\in J, \left| (P_{B_k} - P)\cL_f \right| \leq (1/4) \bar{r}_2^2(\gamma) \bigg \}
\end{equation}where $(\cL_F)_{\bar{r}_2^2(\gamma)} =\{f\in F: P\cL_{f}\leq \bar{r}_2^2(\gamma)\}$. Our first goal is to show that on the event $\Omega_K^\prime$, $P\cL_{\hat f}\leq (1/4) \bar{r}_2^2(\gamma)$. We will then handle  $\bP[\Omega_K^\prime]$.

\begin{Lemma}\label{lem:convexity_of_localized_ball}
Grant Assumptions~\ref{assum:lip_conv} and~\ref{assum:convex}. For every $r\geq0$, the set $(\cL_F)_r:=\{f\in F:P\cL_f\leq r\}$ is convex and relatively closed to $F$ in $L_1(\mu)$. Moreover, if $f\in F$ is such that $P\cL_f>r$ then there exists $f_0\in F$ and $(P\cL_f/r)\geq \alpha>1$ such that $(f-f^*)=\alpha (f_0-f^*)$ and $P\cL_{f_0} = r$.
\end{Lemma}
\begin{proof}
Let $f$ and $g$ be in $(\cL_F)_r$ and $0\leq \alpha\leq1$. We have $\alpha f + (1-\alpha)g\in F$ because $F$ is convex and for all $x\in\cX$ and $y\in\R$, using the convexity of $u\to \bar\ell(u+f^*(x), y)$, we have
\begin{align*}
&\ell_{\alpha f + (1-\alpha)g}(x,y) - \ell_{f^*}(x,y) = \bar\ell(\alpha (f-f^*)(x) + (1-\alpha)(g-f^*)(x) + f^*(x), y) - \bar\ell(f^*(x),y)\\
&\leq \alpha \big(\bar\ell((f-f^*)(x)+ f^*(x), y) - \bar\ell(f^*(x),y)\big) + (1-\alpha)\big(\bar\ell((g-f^*)(x) + f^*(x), y) - \bar\ell(f^*(x),y)\big)\\
&=\alpha(\ell_f - \ell_{f^*}) + (1-\alpha)(\ell_g-\ell_{f^*})
\end{align*}and so $P\cL_{\alpha f + (1-\alpha)g}\leq \alpha P\cL_f + (1-\alpha)P\cL_g$. Given that $P\cL_f, P\cL_g\leq r$ we also have $P\cL_{\alpha f + (1-\alpha)g}\leq r$. Therefore, $\alpha f + (1-\alpha)g\in (\cL_F)_r$ and $(\cL_F)_r$ is convex. 

For all $f,g\in F$, $|P\cL_f - P\cL_g|\leq \norm{f-f^*}_{L_1(\mu)}$ so that $f\in F\to P\cL_f$ is  continuous onto $F$ in $L_1(\mu)$ and therefore its level sets, such as $(\cL_F)_r$, are relatively closed to $F$ in $L_1(\mu)$.

Finally, let $f\in F$ be such that $P\cL_f >r$. Define $\alpha_0 = \sup\{\alpha\geq0: f^*+\alpha(f-f^*)\in(\cL_F)_r\}$. Note that $P\cL_{f^*+\alpha(f-f^*)}\leq \alpha P\cL_f= r$ for $\alpha = r/P\cL_f$ so that $\alpha_0\geq r/P\cL_f$.  Since $(\cL_F)_r$ is relatively closed to $F$ in $L_1(\mu)$, we have $f^*+\alpha_0(f-f^*)\in(\cL_F)_r$ and in particular $\alpha_0<1$ otherwise, by convexity of $(\cL_F)_r$, we would have $f\in (\cL_F)_r$. Moreover, by maximality of $\alpha_0$,  $f_0 = f^*+\alpha_0(f-f^*)$ is such that $P\cL_{f_0}=r$ and the results follows for $\alpha = \alpha_0^{-1}$.
\end{proof}

\begin{Lemma}\label{lem:determinist_part_1_without_ass}
Grant Assumptions~\ref{assum:lip_conv} and~\ref{assum:convex}.  On the event $\Omega_K^\prime$, $P\cL_{\hat f}\leq  \bar{r}_2^2(\gamma)$.
\end{Lemma}
\begin{proof}
Let $f\in F$ be such that $P\cL_f > \bar{r}_2^2(\gamma)$. It follows from Lemma~\ref{lem:convexity_of_localized_ball} that there exists $\alpha\geq 1$ and $f_0\in F$ such that $P\cL_{f_0} = \bar{r}_2^2(\gamma)$ and $f-f^* = \alpha(f_0-f^*)$. According to \eqref{conv_arg_MOM}, we have for every $k\in\{1, \ldots, K\}$, $P_{B_k} \cL_f\geq \alpha P_{B_k}\cL_{f_0}$. Since $f_0\in (\cL_F)_{\bar{r}_2^2(\gamma)}$, on the event $\Omega_K^\prime$, there are strictly more than $K/2$ blocks $B_k$ such that $P_{B_k}\cL_{f_0}\geq P\cL_{f_0}- (1/4) \bar{r}_2^2(\gamma) = (3/4)\bar{r}_2^2(\gamma)$ and so $P_{B_k}\cL_{f}\geq (3/4)\bar{r}_2^2(\gamma)$. As a consequence, we have
\begin{equation}\label{eq:first_ineq_2}
\sup_{f \in  F \backslash (\cL_F)_{\bar{r}_2^2(\gamma)}}  \hspace{0.2cm} \MOM{K}{\ell_{f^{*}}-\ell_f} \leq (-3/4)  \bar{r}_2^2(\gamma)\enspace.
\end{equation}Moreover, on the event $\Omega_K^\prime$, for all $f\in (\cL_F)_{\bar{r}_2^2(\gamma)}$, there are strictly more than $K/2$ blocks $B_k$ such that $P_{B_k}(-\cL_f)\leq (1/4) \bar{r}_2^2(\gamma) - P\cL_f\leq (1/4) \bar{r}_2^2(\gamma)$. Therefore, 
\begin{equation}\label{eq:second_ineq_2}
\sup_{f\in(\cL_F)_{\bar{r}_2^2(\gamma)}}\MOM{K}{\ell_{f^{*}}-\ell_f} \leq (1/4) \bar{r}_2^2(\gamma) \enspace.
\end{equation}We conclude from \eqref{eq:first_ineq_2} and \eqref{eq:second_ineq_2} that
$\sup_{f\in F} \MOM{K}{\ell_{f^{*}}-\ell_f} \leq (1/4) \bar{r}_2^2(\gamma)$ and that every $f\in F$ such that $P\cL_{f}> \bar{r}_2^2(\gamma)$ satisfies  $\MOM{K}{\ell_{f}-\ell_{f^*}}\geq (3/4)\bar{r}_2^2(\gamma)$. But, by definition of $\hat f$, we have
\begin{equation*}
 \MOM{K}{\ell_{\hat f}-\ell_{f^*}}\leq \sup_{f\in F} \MOM{K}{\ell_{f^{*}}-\ell_f} \leq (1/4) \bar{r}_2^2(\gamma) \enspace.
 \end{equation*} Therefore, we necessarily have $P\cL_{\hat f}\leq \bar{r}_2^2(\gamma)$.
 \end{proof}

Now, we prove that $\Omega_K^\prime$ is an exponentially large event using similar argument as in Proposition~\ref{prop:sto_MOM}.

\begin{Proposition}\label{prop:sto_MOM_2}
	Grant Assumptions~\ref{assum:lip_conv},~\ref{assum:convex} and~\ref{assum:moments2} and assume that $(1-\beta)K\geq |\cO|$ and $\beta(1-1/12-32\gamma L)>1/2$. Then $\Omega_K^\prime$ holds with probability larger than $1-\exp(-\beta K/1152)$.
\end{Proposition}

\textit{Sketch of proof.} The proof of Proposition~\ref{prop:sto_MOM_2} follows the same line as the one of Proposition~\ref{prop:sto_MOM}. Let us precise the main differences. We set $\cF^\prime = (\cL_F)_{\bar{r}_2^2(\gamma)}$ and for all $f\in \cF^\prime$, $z^\prime(f) = \sum_{k=1}^K I\{|G_f(W_k)|\leq (1/4) \bar{r}_2^2(\gamma)\}$ where $G_f(W_k)$ is the same quantity as in the proof of  Proposition~\ref{prop:sto_MOM_2}. Let us consider the contraction $\phi$ introduced in Proposition~\ref{prop:sto_MOM_2}. By definition of $\bar{r}_2^2(\gamma)$ and $V_K(\cdot)$, we have
	\begin{align*}
	\mathbb{E}\phi(8 & (\bar{r}_2^2(\gamma))^{-1} | G_f(W_k)|)  \leq \mathbb{P} \bigg( |G_f(W_k)| \geq \frac{\bar{r}_2^2(\gamma)}{8} \bigg) \leq \frac{64}{(\bar{r}_2^2(\gamma))^2} \mathbb{E}G_f(W_k)^2  = \frac{64}{ (\bar{r}_2^2(\gamma))^2} \mathbb{V}ar (P_{B_k}\cL_f) \\
	&  \leq  \frac{64K^2}{(\bar{r}_2^2(\gamma))^2N^2} \sum_{i \in B_k} \mathbb{V}ar_{P_i}(\cL_f) \leq \frac{64K}{(\bar{r}_2^2(\gamma))^2N} \sup\{\mathbb{V}ar_{P_i}(\cL_f):f\in\cF^\prime, i\in\cI\} \\
	& \leq \frac{64K}{(\bar{r}_2^2(\gamma))^2N} \sup\{\mathbb{V}ar_{P_i}(\cL_f):P\cL_f \leq \bar{r}_2^2(\gamma), i\in\cI\}  \leq \frac{1}{24} \enspace.
	\end{align*}
Using Mc Diarmid's inequality, the Gin{\'e}-Zinn symmetrization argument and the contraction lemma twice and the Lipschitz property of the loss function,  such as in the proof of 	Proposition~\ref{prop:sto_MOM}, we obtain with probability larger than $1-\exp(-|\cK|/1152)$, for all $f\in\cF^\prime$,
\begin{equation}\label{eq:concentration_McDirmid_without_ass}
z(f)\geq |\cK|(1-1/12) -\frac{32LK}{ N} \E \sup_{f\in\cF^\prime} \frac{1}{\bar{r}_2^2(\gamma)}\left|\sum_{i\in\cup_{k\in\cK}B_k} \sigma_i (f-f^*)(X_i)\right|. 
\end{equation} 

Now, it remains to use the definition of $\bar{r}_2^2(\gamma)$ to bound the expected supremum in the right-hand side of \eqref{eq:concentration_McDirmid_without_ass} to get
\begin{equation}\label{eq:sup_ball_2}
\E \sup_{f\in\cF^\prime} \frac{1}{\bar{r}_2(\gamma)^2}\left|\sum_{i\in\cup_{k\in\cK}B_k} \sigma_i (f-f^*)(X_i)\right|\leq \frac{\gamma |\cK|N}{K}.
\end{equation}

\textbf{Proof of Theorem~\ref{theo:main_without_bernstein_cond}}. The proof of Theorem~\ref{theo:main_without_bernstein_cond} follows from Lemma~\ref{lem:determinist_part_1_without_ass} and Proposition~\ref{prop:sto_MOM_2} for $\beta=4/7$ 
and $\gamma = 1/(768L)$.

\section{Proof of Lemma~\ref{Lem:LinReg} \label{lem:rad}}

\begin{proof}
We have
  \begin{align*}
  \frac{1}{\sqrt{N}} \mathbb{E} \sup_{ f \in F: \|f-f^{*}\|_{L_2 } \leq r} \sum_{i=1}^N \sigma_i  (f-f^{*})(X_i) 
  & =  \mathbb{E} \sup_{ t \in \mathbb{R}^d: \mathbb{E} \inr{t,X}^2 \leq r^2  }   \inr{t,   \frac{1}{\sqrt{N}} \sum_{i=1}^N \sigma_i X_i} \enspace.
  \end{align*}
  Let $\Sigma=\mathbb{E} X^TX$ denote the covariance matrix of $X$ and consider its SVD, $\Sigma = QDQ^T$ where $Q = [Q_1|\cdots|Q_d]\in\bR^{d\times d}$ is an orthogonal matrix and $D$ is a diagonal $d\times d$ matrix with non-negative entries. For all $t\in\R^d$, we have $\mathbb{E}\inr{X,t}^2 = t^T \Sigma t = \sum_{j=1}^d d_j \inr{t,Q_j}^2$.  Then
  \begin{align*}
  &\mathbb{E} \sup_{ t \in \mathbb{R}^d: \sqrt{\mathbb{E} \inr{t,X}^2 } \leq r  } \inr{t,  \frac{1}{\sqrt{N}} \sum_{i=1}^N \sigma_i X_i}  =   \mathbb{E} \sup_{ t \in \mathbb{R}^d: \sqrt{\mathbb{E} \inr{t  ,X}^2 } \leq r  }  \inr{ \sum_{j=1}^d \inr{t,Q_j}Q_j ,   \frac{1}{\sqrt{N}} \sum_{i=1}^N \sigma_i X_i }\\
  &=  \mathbb{E} \sup_{ t \in \mathbb{R}^d: \sqrt{ \sum_{j=1}^d d_j \inr{t,Q_j}^2 } \leq r  } \sum_{j=1:d_j\ne 0}^d \sqrt{d_j} \inr{t,Q_j}  \inr{ \frac{Q_j}{\sqrt{d_j}},   \frac{1}{\sqrt{N}} \sum_{i=1}^N \sigma_i X_i } \\ 
  & \leq r \mathbb{E} \sqrt{ \sum_{j=1:d_j\ne 0}^d \inr{ \frac{Q_j}{\sqrt{d_j}},  \frac{1}{\sqrt{N}} \sum_{i=1}^N \sigma_i X_i }^2} \leq  r \sqrt{ \mathbb{E}  \sum_{j=1:d_j\ne 0}^d \inr{ \frac{Q_j}{\sqrt{d_j}},   \frac{1}{\sqrt{N}} \sum_{i=1}^N \sigma_i X_i }^2} \enspace.
  \end{align*}  
Moreover, for any $j$ such that $d_j\ne 0$,
  \begin{align*}
   \mathbb{E} \inr{ \frac{Q_j}{\sqrt{d_j}}, & \frac{1}{\sqrt{N}} \sum_{i=1}^N \sigma_i X_i }^2   = \mathbb{E} \frac{1}{N} \sum_{k,l=1}^N \sigma_l \sigma_k \inr{ \frac{Q_j}{\sqrt{d_j}}, X_k }\inr{ \frac{Q_j}{\sqrt{d_j}}, X_l} = \frac{1}{N} \sum_{k=1}^N \mathbb{E}  \inr{ \frac{Q_j}{\sqrt{d_j}}, X_k}^2 \\
   & =  \frac{1}{N} \sum_{k=1}^N  \bigg(\frac{Q_j}{\sqrt{d_j}}\bigg)^T \mathbb{E} X_k^TX_k \bigg(\frac{Q_j}{\sqrt{d_j}}\bigg) = \frac{1}{N} \sum_{k=1}^N  \bigg(\frac{Q_j}{\sqrt{d_j}}\bigg)^T \Sigma\bigg(\frac{Q_j}{\sqrt{d_j}}\bigg) 
  \end{align*}
By orthonormality, $Q^TQ_j  = e_j$ and $Q_j^TQ = e_j^T$, then, for any $j$ such that $d_j\ne 0$, 
  \begin{align*}
  \mathbb{E} \inr{ \frac{Q_j}{\sqrt{d_j}},   \frac{1}{\sqrt{N}} \sum_{i=1}^N \sigma_i X_i }^2 = \frac{1}{N} \sum_{k=1}^N \frac{1}{d_j} e_j^T D e_j = 1\enspace.
  \end{align*} 
Finally, we obtain
  \begin{align*}
    \frac{1}{\sqrt{N}} \mathbb{E} \sup_{ f \in F: \|f-f^{*}\|_{L_2 } \leq r} \sum_{i=1}^N \sigma_i  (f-f^{*})(X_i) \leq r \sqrt{\sum_{j=1}^d {\bf 1}_{\{d_j\ne 0\}}} = r \sqrt{\text{Rank}(\Sigma)}
  \end{align*}
and therefore the fixed point $\tilde{r}_2(\gamma)$ is such that
  \begin{align*}
  \tilde{r}_2(\gamma)  = & \inf  \bigg\{ r > 0,  \forall J \in \mathcal{I}: |J| \geq N/2, \hspace{0.2cm} \mathbb{E}\sup_{t \in \mathbb{R}^d : \sqrt{\mathbb{E} \inr{t-t^{*},X}^2 } \leq r }  \sum_{i \in J }  \sigma_i  \inr{X_i,t-t^{*}} \hspace{0.1cm}  \leq r^2 |J| \gamma \bigg\}\\
  &  \leq   \inf  \bigg\{ r > 0,  \forall J \in \mathcal{I}: |J| \geq N/2, \hspace{0.2cm}\hspace{0.2cm} r\sqrt{\text{Rank}(\Sigma)}  \leq r^2 \sqrt{|J|} \gamma \bigg\} \leq \sqrt{\frac{\text{Rank}(\Sigma)}{2\gamma^2N}}\enspace.
  \end{align*}
 \end{proof}

\section{Proofs of the results of Section~\ref{app:ass}} \label{proof_bernstein}
We begin this Section with a simple Lemma coming from the convexity of $F$.
\begin{Lemma} \label{lemm:ber}
	For any $f \in F$,
	\begin{equation*}
		\lim_{t \rightarrow 0^+} \frac{R(f^*+t(f-f^*)) - R(f^*)}{t} \geq 0
	\end{equation*}
	where we recall that $R(f) = \E_{(X,Y) \sim P} [\ell_f(X,Y)]$.
\end{Lemma}
\begin{proof}
	Let $t \in (0,1)$. By convexity of $F$, $f^* + t(f-f^*) \in F$ 	and $R(f^*+t(f-f^*)) - R(f^*) \geq  0$ because $f^*$ minimizes the risk over $F$. 
\end{proof}

\subsection{Proof of Theorem~ \ref{quantile_loss}}\label{proof:thm4}
Let $r >0$. Let $f\in F$ be such that $\norm{f-f^*}_{L_2}\leq r$. For all $x\in\cX$ denote by $F_{Y|X=x}$ the conditional c.d.f. of $Y$ given $X=x$. We have
\begin{align*}
\mathbb{E} \bigg[  \ell_f(X,Y) | X = x \bigg] & = (\tau -1) \int {\bf 1}_{y \leq f(x)} (y-f(x))F_{Y|X=x}({\rm d}y) + \tau \int {\bf 1}_{y > f(x)} (y-f(x))F_{Y|X=x}({\rm d}y) \\
& =   \int {\bf 1}_{y > f(x)} (y-f(x))F_{Y|X=x}({\rm d}y) + (\tau-1) \int {\bf 1}_{\mathbb{R}} (y-f(x))F_{Y|X=x}({\rm d}y)\enspace.
\end{align*}
By Fubini's theorem,
\begin{align*}
&\int {\bf 1}_{z \geq f(x)} (1-F_{Y|X=x}(z))\text{d}z  =  \int {\bf 1}_{z \geq f(x)}\bigg( 1 - \bP(Y \leq z |X=x ) \bigg) {\rm d}z   =  \int {\bf 1}_{z \geq f(x)} \mathbb{E} [ {\bf 1}_{  Y > z }|X=x ] \text{d}z\\
&  = \int \int  {\bf 1}_{  y > z \geq f(x) } f_{Y|X=x}(y)\text{d}y \text{d}z = \int  {\bf 1}_{  y > f(x) }(y-f(x)) f_{Y|X=x}(y) \text{d}y\\
&  = \int {\bf 1}_{y > f(x)} (y-f(x)) F_{Y|X=x}({\rm d}y)\enspace.
\end{align*}
Therefore,
\begin{align*}
\mathbb{E} \bigg[  \ell_f(X,Y) | X=x\bigg] & =  \int {\bf 1}_{y \geq f(x)} (1-F_{Y|X=x}(y))\text{d}y + (\tau -1) \bigg( \int_{\mathbb{R}} yF_{Y|X=x}({\rm d}y) - f(x) \bigg) \\
& = g(x,f(x)) + (\tau-1 )  \int_{\mathbb{R}} yF_{Y|X=x}({\rm d}y)
\end{align*}
where $g:(x,a)\in\cX\times\bR \to \int {\bf 1}_{y \geq a} (1 - F_{Y|X=x}(y))\text{d}y + (1-\tau)a$. It follows that 
\begin{equation}\label{eq:Risk1}
P\mathcal{L}_f=\E[g(X,f(X))-g(X,f^*(X))]\enspace. 
\end{equation} 
Since for all $x \in \cX$, $a \mapsto g(x,a)$ is twice differentiable, from a second order Taylor expansion we get
\begin{align*}
P\cL_f = \E \bigg[  g(X,f(X)) - g(X,f^*(X)) \bigg] & = \E \bigg[  \frac{\partial g(X,a)}{\partial a}(f^*(X)) (f(X)-f^*(X)) \bigg] \\
& + \frac{1}{2} \int_{x \in \cX} \frac{\partial^2 g(x,a)}{\partial a^2}(z_x) (f(x)-f^*(x))^2 dP_X(x)
\end{align*}
where for all $x\in\cX$, $z_x$ is some point in $\big [\min(f(x), f^{*}(x)),\max(f(x), f^{*}(x)) \big]$. For the first order term, we have 
\begin{equation*}
  \E \bigg[  \frac{\partial g(X,a)}{\partial a}(f^*(X)) (f(X)-f^*(X)) \bigg] = \E \lim_{t \rightarrow 0^+} \frac{g(X,f^*(X)+ t(f(X)-f^*(X)) - g(X,f^*(X))}{t}.
\end{equation*}
For all $x \in \cX$, we have $[g(x,f^*(x)+t(f(x)-f^*(x))) - g(x,f^*(x))]/t \leq  (2-\tau) |f^*(x)- f(x)|$ which is integrable with respect to $P_X$. Thus, by the dominated convergence theorem, it is possible to interchange integral and limit and therefore using  Lemma~\ref{lemm:ber}, we obtian 
\begin{align*}
\E \bigg[  \frac{\partial g(X,a)}{\partial a}(f^*(X)) (f(X)-f^*(X)) \bigg] & =  \lim_{t \rightarrow 0^+} \bE \frac{g(X,f^*(X)+ t(f(X)-f^*(X)) - g(X,f^*(X))}{t} \\
& = \lim_{t \rightarrow 0^+} \frac{R(f^*+ t(f-f^*)) - R(f^*)}{t} \geq 0.
\end{align*}
Given that for all $x\in\cX$, $\frac{\partial^2 g (x,a)}{\partial a^2} (z)= f_{Y|X=x}(z)$ for all $z\in \bR$ it follows that
\begin{align*}
P\cL_f  \geq  \frac{1}{2} \int_{x \in \cX} f_{Y|X=x}(z_x) (f(x)-f^*(x))^2 dP_X(x).
\end{align*}

Consider $A = \{ x \in \mathcal{X}, |f(x)-f^{*}(x)| \leq   (\sqrt{2}C')^{(2+\varepsilon)/\varepsilon} r  \}$. Given that $\|f-f^{*}\|_{L_2 }\leq r$, by Markov's inequality, $P(X \in A) \geqslant  1-1/(\sqrt{2}C')^{(4+2\varepsilon)/\varepsilon}$. From Assumption~\ref{ass:quantil} we get
\begin{align}\label{eq:Risk2_2} 
\frac{2P\mathcal{L}_f}{\alpha} & \geqslant \mathbb{E} [  I_A(X) (f(X)-f^{*}(X))^2 ] =\|f-f^*\|_{L_2 }^2-\mathbb{E} [  I_{A^c}(X) (f(X)-f^{*}(X))^2 ] \enspace.
\end{align}
By Holder and Markov's inequalities,
\[
\mathbb{E} [  I_{A^c}(X) (f(X)-f^{*}(X))^2 ]\leqslant  \big(  \mathbb{E} [  I_{A^c}(X)]   \big)^{\varepsilon/(2+\varepsilon)}  \big( \mathbb{E} [  (f(X)-f^{*}(X))^{2+\varepsilon} ]  \big)^{2/(2+\varepsilon)} \leqslant \frac{ \|f-f^{*}\|_{L_{2+\varepsilon}}^2}{2(C')^2}\enspace.
\]
By Assumption~\ref{ass:L4L2}, it follows that  $\mathbb{E} [  I_{A^c}(X) (f(X)-f^{*}(X))^2 ]\leqslant  \|f-f^{*}\|_{L_2}^2/2$ and we conclude with \eqref{eq:Risk2_2}.

\subsection{Proof of Theorem~\ref{huber_loss}}\label{proof:thm5}
Let $r>0$. Let $f\in F$ be such that $\norm{f-f^*}_{L_2}\leq r$. We have
\begin{align*}
P\cL_f = \mathbb{E}_X \mathbb{E} \bigg[  \rho_H(Y-f(x)) - \rho_H(Y-f^*(x)) | X= x \bigg] = \E \big[g(X, f(X)) - g(X, f^*(X))\big]
\end{align*}where  $g:(x,a)\in\cX\times\bR = \mathbb{E}[\rho_H(Y-a)|X=x]$. Let $F_{Y|X=x}$ denote the c.d.f. of $Y$ given $X=x$. Since for all $x \in \cX$, $a \mapsto g(x,a)$ is twice differentiable in its second argument (see Lemma 2.1 in \cite{elsener2016robust}), a second Taylor expansion yields
\begin{align*}
 P\cL_f & = \E \bigg[ \frac{\partial g(X,a)}{\partial a}(f^*(X))(f(X)-f^*(X)) \bigg]  +  \frac{1}{2} \int_{x \in \cX}  (f(x)-f^{*}(x))^2  \frac{\partial^2 g(x,a)}{\partial a^2}(z_x)  dP_X(x)  
\end{align*}
 where for all $x\in\cX$, $z_x$ is some point in $[\min(f(x), f^{*}(x)), \max(f(x), f^{*}(x))]$. By Lemma~\ref{lemm:ber}, with the same reasoning as the one in Section~\ref{proof:thm4}, we get
\begin{align*}
P\cL_f  \geq \frac{1}{2} \int_{x \in \cX}  (f(x)-f^{*}(x))^2  \frac{\partial^2 g(x,a)}{\partial a^2}(z_x)  dP_X(x)  \enspace.
\end{align*}
Moreover, for all $z \in\bR$,
\begin{align*}
\frac{\partial^2 g(x,a)}{\partial a^2}(z) &  = F_{Y|X=x}(z + \delta) - F_{Y|X=x}(z-\delta).
\end{align*}
Now, let $A = \{ x \in \cX: |f(x)-f^{*}(x)| \leq (\sqrt{2}C')^{(2+\varepsilon)/\varepsilon} r   \} $. It follows from Assumption \ref{ass:quantil} that $P\cL_f \geq  (\alpha/2) \mathbb{E} [(f(X)-f^{*}(X))^2 I_A(X)]$. Since $\|f-f^{*}\|_{L_2 }\leq r$, by Markov's inequality, $P(X \in A) \geqslant  1-1/(\sqrt{2}C')^{(4+2\varepsilon)/\varepsilon}$. By Holder and Markov's inequalities,
\[
\mathbb{E} [  I_{A^c}(X) (f(X)-f^{*}(X))^2 ]\leqslant  \big(  \mathbb{E} [  I_{A^c}(X)]   \big)^{\varepsilon/(2+\varepsilon)}  \big( \mathbb{E} [  (f(X)-f^{*}(X))^{2+\varepsilon} ]  \big)^{2/(2+\varepsilon)} \leqslant \frac{ \|f-f^{*}\|_{L_{2+\varepsilon}}^2}{2(C')^2}\enspace.
\]
By Assumption~\ref{ass:L4L2}, it follows that  $\mathbb{E} [  I_{A^c}(X) (f(X)-f^{*}(X))^2 ]\leqslant  \frac{\|f-f^{*}\|_{L_2}^2}{2}$, which concludes the proof.

\subsection{Proof of Theorem~\ref{thm:AppliLog}}\label{proof:thm7}
Let $r>0$. Let $f\in F$ be such that $\norm{f-f^*}_{L_2}\leq r$. Let $\eta(x) = P(Y=1|X=x)$. Write first that $P\cL_f = \mathbb{E} \bigg[  g(X, f(X)) - g(X, f^*(X))\bigg]$ where for all $x\in\cX$ and $a\in\bR$,  $g(x,a) = \eta(x) \log(1+\exp(-a)) + (1-\eta(x))\log(1+\exp(a))$. From Lemma~\ref{lemm:ber} and the same reasoning as in Section~\ref{proof:thm4} and~\ref{proof:thm5} we get 
\begin{align*}
P\cL_f \geq   \int_{x \in \cX} \frac{\partial_2^2 g(x,a)}{\partial a^2}(z_x)\frac{(f(x)-f^{*}(x))^2 }{2} dP_X(x) = \int_{x \in \cX} \frac{e^{z_x}}{(1+e^{z_x})^2}\frac{(f(x)-f^{*}(x))^2 }{2} dP_X(x)
\end{align*}
for some $z_x \in [\min(f(x), f^{*}(x)), \max(f(x), f^{*}(x))]$. Now, let
\begin{equation*}
	A = \left\{ x \in \mathcal{X}: |f^{*}(x)|\leq c_0, |f(x)-f^{*}(x)|\leq (2C')^{(2+\varepsilon)/\varepsilon} r  \right\} \enspace.
\end{equation*}
On the event $A$ we have
\begin{align*}
P\mathcal{L}_f  \geq \frac{e^{- c_0 -(2C')^{(2+\varepsilon)/\varepsilon} r  }  }{2\big( 1+ e^{c_0 + (2C')^{(2+\varepsilon)/\varepsilon} r  }  \big)^2 }  \mathbb{E} [I_A(X) (f(X)-f^{*}(X))^2]
\end{align*}
Using the fact that $P(X \notin A) \leq P(|f^*(X) | > c_0) + P(|f(X)-f^*(X| > (2C')^{(2+\varepsilon)/\varepsilon} r ) \leq 2/(2C')^{(4+\varepsilon)/\varepsilon}  $, we conclude with Assumption \ref{ass:L4L2} and the same analysis as in the two previous proofs.

\subsection{Proof of Theorem~\ref{thm:hinge}} \label{proof:thm8}

Let $r>0$ such that $ r(\sqrt 2 C')^{(2+\varepsilon)/\varepsilon} \leq 1$. Let $f$ be in $F$ such that $\|f-f^*\|_{L_2} \leq r$. For all $x$ in $\mathcal{X}$ let us denote $\eta(x) = \mathbb P(Y=1 | X=x) $. It is easy to verify that the Bayes estimator (which is equal to the oracle) is defined as $f^*(x) = \mbox{sign}(2\eta(x)-1)$. Consider the set $A= \{ x \in \mathcal X, |f(x)-f^*(x)| \leq r(\sqrt 2 C')^{(2+\varepsilon)/\varepsilon}  \}$. Since $\|f-f^*\|_{L_2} \leq r$, by Markov's inequality $\mathbb P(X \in A) \geq 1-1/(\sqrt 2 C')^{(4+ 2\varepsilon)/\varepsilon}$. Let $x$ be in $A$. If $f^*(x) = -1$ (i.e $2\eta(x) \leq 1$) and $f(x) \leq f^*(x) = -1$ we obtain
\begin{equation*}
	\mathbb E \big[ \ell_f(X,Y) | X= x  \big]- \mathbb E \big[ \ell_{f^*}(X,Y) | X= x  \big] = \eta(x)(1-f(x)) - \eta(x) (1-f^*(x)) \geq \eta(x) \big(f(x)-f^*(x) \big)^2
\end{equation*}
where we used the fact that on $A$, $|f(x)-f^*(x)| \leq  r(\sqrt 2 C')^{(2+\varepsilon)/\varepsilon}  \leq 1$. 
Using the same analysis for the other cases we get that 
\begin{align*}
	\mathbb E \big[ \ell_f(X,Y) | X= x  \big]- \mathbb E \big[ \ell_{f^*}(X,Y) | X= x  \big] & \geq \min \big(\eta(x),1-\eta(x), |1-2\eta(x)| \big) \big(f(x)-f^*(x) \big)^2 \\
	&  \geq \alpha \big(f(x)-f^*(x) \big)^2 
\end{align*}
Therefore,
\begin{align}\label{eq:Risk2} 
	\frac{P\mathcal{L}_f}{\alpha} & \geqslant \mathbb{E} [  I_A(X) (f(X)-f^{*}(X))^2 ] =\|f-f^*\|_{L_2 }^2-\mathbb{E} [  I_{A^c}(X) (f(X)-f^{*}(X))^2 ] \enspace.
\end{align}
By Holder and Markov's inequalities,
\[
\mathbb{E} [  I_{A^c}(X) (f(X)-f^{*}(X))^2 ]\leqslant  \big(  \mathbb{E} [  I_{A^c}(X)]   \big)^{\varepsilon/(2+\varepsilon)}  \big( \mathbb{E} [  (f(X)-f^{*}(X))^{2+\varepsilon} ]  \big)^{2/(2+\varepsilon)} \leqslant \frac{ \|f-f^{*}\|_{L_{2+\varepsilon}}^2}{2(C')^2}\enspace.
\]
By Assumption~\ref{ass:L4L2}, it follows that  $\mathbb{E} [  I_{A^c}(X) (f(X)-f^{*}(X))^2 ]\leqslant  \frac{\|f-f^{*}\|_{L_2}^2}{2}$ and we conclude with \eqref{eq:Risk2}.
%


\begin{footnotesize}
\bibliographystyle{plain}
\bibliography{biblio}

\begin{thebibliography}{10}

\bibitem{MR1688610}
Noga Alon, Yossi Matias, and Mario Szegedy.
\newblock The space complexity of approximating the frequency moments.
\newblock {\em J. Comput. System Sci.}, 58(1, part 2):137--147, 1999.
\newblock Twenty-eighth Annual ACM Symposium on the Theory of Computing
  (Philadelphia, PA, 1996).

\bibitem{pierre2017estimation}
P.~Alquier, V.~Cottet, and G.~Lecu{\'e}.
\newblock Estimation bounds and sharp oracle inequalities of regularized
  procedures with lipschitz loss functions.
\newblock {\em arXiv preprint arXiv:1702.01402}, 2017.

\bibitem{MR2906886}
Jean-Yves Audibert and Olivier Catoni.
\newblock Robust linear least squares regression.
\newblock {\em Ann. Statist.}, 39(5):2766--2794, 2011.

\bibitem{bach2011convex}
Francis Bach, Rodolphe Jenatton, Julien Mairal, Guillaume Obozinski, et~al.
\newblock Convex optimization with sparsity-inducing norms.
\newblock {\em Optimization for Machine Learning}, 5:19--53, 2011.

\bibitem{MR3595933}
Y.~Baraud, L.~Birg\'e, and M.~Sart.
\newblock A new method for estimation and model selection: {$\rho$}-estimation.
\newblock {\em Invent. Math.}, 207(2):425--517, 2017.

\bibitem{MR2166554}
Peter~L. Bartlett, Olivier Bousquet, and Shahar Mendelson.
\newblock Local {R}ademacher complexities.
\newblock {\em Ann. Statist.}, 33(4):1497--1537, 2005.

\bibitem{bartlett2005local}
Peter~L Bartlett, Olivier Bousquet, Shahar Mendelson, et~al.
\newblock Local rademacher complexities.
\newblock {\em The Annals of Statistics}, 33(4):1497--1537, 2005.

\bibitem{MR2240689}
Peter~L. Bartlett and Shahar Mendelson.
\newblock Empirical minimization.
\newblock {\em Probab. Theory Related Fields}, 135(3):311--334, 2006.

\bibitem{MR762855}
Lucien Birg\'e.
\newblock Stabilit\'e et instabilit\'e du risque minimax pour des variables
  ind\'ependantes \'equidistribu\'ees.
\newblock {\em Ann. Inst. H. Poincar\'e Probab. Statist.}, 20(3):201--223,
  1984.

\bibitem{MR2182250}
St\'{e}phane Boucheron, Olivier Bousquet, and G\'{a}bor Lugosi.
\newblock Theory of classification: a survey of some recent advances.
\newblock {\em ESAIM Probab. Stat.}, 9:323--375, 2005.

\bibitem{MR3185193}
St\'{e}phane Boucheron, G\'{a}bor Lugosi, and Pascal Massart.
\newblock {\em Concentration inequalities}.
\newblock Oxford University Press, Oxford, 2013.
\newblock A nonasymptotic theory of independence, With a foreword by Michel
  Ledoux.

\bibitem{bubeck2015convex}
S{\'e}bastien Bubeck.
\newblock Convex optimization: Algorithms and complexity.
\newblock {\em Foundations and Trends{\textregistered} in Machine Learning},
  8(3-4):231--357, 2015.

\bibitem{MR3052407}
Olivier Catoni.
\newblock Challenging the empirical mean and empirical variance: a deviation
  study.
\newblock {\em Ann. Inst. Henri Poincar\'{e} Probab. Stat.}, 48(4):1148--1185,
  2012.

\bibitem{devroye2016sub}
Luc Devroye, Matthieu Lerasle, Gabor Lugosi, Roberto~I Oliveira, et~al.
\newblock Sub-gaussian mean estimators.
\newblock {\em The Annals of Statistics}, 44(6):2695--2725, 2016.

\bibitem{elsener2016robust}
Andreas Elsener and Sara van~de Geer.
\newblock Robust low-rank matrix estimation.
\newblock {\em arXiv preprint arXiv:1603.09071}, 2016.

\bibitem{jon}
Qiyang Han and Jon~A. Wellner.
\newblock Convergence rates of least squares regression estimators with
  heavy-tailed errors.
\newblock 2017.

\bibitem{huber2011robust}
Peter~J Huber and E.~Ronchetti.
\newblock Robust statistics.
\newblock In {\em International Encyclopedia of Statistical Science}, pages
  1248--1251. Springer, 2011.

\bibitem{MR855970}
Mark~R. Jerrum, Leslie~G. Valiant, and Vijay~V. Vazirani.
\newblock Random generation of combinatorial structures from a uniform
  distribution.
\newblock {\em Theoret. Comput. Sci.}, 43(2-3):169--188, 1986.

\bibitem{MR2329442}
Vladimir Koltchinskii.
\newblock Local {R}ademacher complexities and oracle inequalities in risk
  minimization.
\newblock {\em Ann. Statist.}, 34(6):2593--2656, 2006.

\bibitem{koltchinskii2011empirical}
Vladimir Koltchinskii.
\newblock Empirical and rademacher processes.
\newblock In {\em Oracle Inequalities in Empirical Risk Minimization and Sparse
  Recovery Problems}, pages 17--32. Springer, 2011.

\bibitem{MR2829871}
Vladimir Koltchinskii.
\newblock {\em Oracle inequalities in empirical risk minimization and sparse
  recovery problems}, volume 2033 of {\em Lecture Notes in Mathematics}.
\newblock Springer, Heidelberg, 2011.
\newblock Lectures from the 38th Probability Summer School held in Saint-Flour,
  2008, \'Ecole d'\'Et\'e de Probabilit\'es de Saint-Flour. [Saint-Flour
  Probability Summer School].

\bibitem{MR3431642}
Vladimir Koltchinskii and Shahar Mendelson.
\newblock Bounding the smallest singular value of a random matrix without
  concentration.
\newblock {\em Int. Math. Res. Not. IMRN}, (23):12991--13008, 2015.

\bibitem{lecue2017learning}
G.~Lecu{\'e} and M.~Lerasle.
\newblock Learning from mom's principles: Le cam's approach.
\newblock {\em To appear in Stochastic processes and their applications}, 2017.

\bibitem{lecue2017robust}
G.~Lecu{\'e} and M.~Lerasle.
\newblock Robust machine learning by median-of-means: theory and practice.
\newblock {\em To appear in the Annals of Statistics}, 2017.

\bibitem{lecue2016performance}
Guillaume Lecu{\'e} and Shahar Mendelson.
\newblock Performance of empirical risk minimization in linear aggregation.
\newblock {\em Bernoulli}, 22(3):1520--1534, 2016.

\bibitem{Lecue2018a}
G.~Lecué, M.~Lerasle, and T.~Mathieu.
\newblock Robust classification via mom minimization.
\newblock {\em arXiv:1808.03106}, 2018.

\bibitem{MR1849347}
Michel Ledoux.
\newblock {\em The concentration of measure phenomenon}, volume~89 of {\em
  Mathematical Surveys and Monographs}.
\newblock American Mathematical Society, Providence, RI, 2001.

\bibitem{ledoux2013probability}
Michel Ledoux and Michel Talagrand.
\newblock {\em Probability in Banach Spaces: isoperimetry and processes}.
\newblock Springer Science \& Business Media, 2013.

\bibitem{LugosiMendelson2016}
Gabor Lugosi and Shahar Mendelson.
\newblock Risk minimization by median-of-means tournaments.
\newblock {\em To appear in JEMS}, 2016.

\bibitem{LugosiMendelson2017}
Gabor Lugosi and Shahar Mendelson.
\newblock Regularization, sparse recovery, and median-of-means tournaments.
\newblock {\em Preprint available on arXiv:1701.04112}, 2017.

\bibitem{LugosiMendelson2017-2}
Gabor Lugosi and Shahar Mendelson.
\newblock Sub-gaussian estimators of the mean of a random vector.
\newblock {\em To appear in Ann. Statist. arXiv:1702.00482}, 2017.

\bibitem{MR1765618}
Enno Mammen and Alexandre~B. Tsybakov.
\newblock Smooth discrimination analysis.
\newblock {\em Ann. Statist.}, 27(6):1808--1829, 1999.

\bibitem{mendelson2014learning}
Shahar Mendelson.
\newblock Learning without concentration.
\newblock In {\em Conference on Learning Theory}, pages 25--39, 2014.

\bibitem{MR3367000}
Shahar Mendelson.
\newblock Learning without concentration.
\newblock {\em J. ACM}, 62(3):Art. 21, 25, 2015.

\bibitem{MR3645129}
Shahar Mendelson.
\newblock On multiplier processes under weak moment assumptions.
\newblock In {\em Geometric aspects of functional analysis}, volume 2169 of
  {\em Lecture Notes in Math.}, pages 301--318. Springer, Cham, 2017.

\bibitem{MR2373017}
Shahar Mendelson, Alain Pajor, and Nicole Tomczak-Jaegermann.
\newblock Reconstruction and subgaussian operators in asymptotic geometric
  analysis.
\newblock {\em Geom. Funct. Anal.}, 17(4):1248--1282, 2007.

\bibitem{MR702836}
A.~S. Nemirovsky and D.~B. Yudin.
\newblock {\em Problem complexity and method efficiency in optimization}.
\newblock A Wiley-Interscience Publication. John Wiley \& Sons, Inc., New York,
  1983.
\newblock Translated from the Russian and with a preface by E. R. Dawson,
  Wiley-Interscience Series in Discrete Mathematics.

\bibitem{scikit-learn}
F.~Pedregosa, G.~Varoquaux, A.~Gramfort, V.~Michel, B.~Thirion, O.~Grisel,
  M.~Blondel, P.~Prettenhofer, R.~Weiss, V.~Dubourg, J.~Vanderplas, A.~Passos,
  D.~Cournapeau, M.~Brucher, M.~Perrot, and E.~Duchesnay.
\newblock Scikit-learn: Machine learning in {P}ython.
\newblock {\em Journal of Machine Learning Research}, 12:2825--2830, 2011.

\bibitem{saumard2018optimality}
Adrien Saumard.
\newblock On optimality of empirical risk minimization in linear aggregation.
\newblock {\em Bernoulli}, 24(3):2176--2203, 2018.

\bibitem{MR3184689}
Michel Talagrand.
\newblock {\em Upper and lower bounds for stochastic processes}, volume~60 of
  {\em Ergebnisse der Mathematik und ihrer Grenzgebiete. 3. Folge. A Series of
  Modern Surveys in Mathematics [Results in Mathematics and Related Areas. 3rd
  Series. A Series of Modern Surveys in Mathematics]}.
\newblock Springer, Heidelberg, 2014.
\newblock Modern methods and classical problems.

\bibitem{MR2051002}
Alexandre~B. Tsybakov.
\newblock Optimal aggregation of classifiers in statistical learning.
\newblock {\em Ann. Statist.}, 32(1):135--166, 2004.

\bibitem{MR3526202}
Sara van~de Geer.
\newblock {\em Estimation and testing under sparsity}, volume 2159 of {\em
  Lecture Notes in Mathematics}.
\newblock Springer, [Cham], 2016.
\newblock Lecture notes from the 45th Probability Summer School held in
  Saint-Four, 2015, \'Ecole d'\'Et\'e de Probabilit\'es de Saint-Flour.
  [Saint-Flour Probability Summer School].

\bibitem{Vapnik:1995:NSL:211359}
Vladimir Vapnik.
\newblock {\em The Nature of Statistical Learning Theory}.
\newblock Springer-Verlag New York, Inc., New York, NY, USA, 1995.

\bibitem{vapnik1998statistical}
Vladimir Vapnik.
\newblock {\em Statistical learning theory}, volume~1.
\newblock Wiley New York, 1998.

\bibitem{fan}
W.-X. Zhou, K.~Bose, J.~Fan, and H.~Liu.
\newblock A new perspective on robust m-estimation: Finite sample theory and
  applications to dependence-adjusted multiple testing.
\newblock 2018.

\end{thebibliography}
\end{footnotesize}


\end{document}